\documentclass[a4paper,reqno,11pt]{amsart}
\usepackage[tmargin=30mm,bmargin=30mm,hmargin=25mm]{geometry}
\usepackage{amssymb}
\usepackage{amstext}
\usepackage{amsmath}
\usepackage{amscd}
\usepackage{amsthm}
\usepackage{amsfonts}
\usepackage{enumerate}
\usepackage{graphicx}
\usepackage{latexsym}
\usepackage{mathrsfs}
\input xy
\xyoption{all}
\usepackage{pstricks}
\usepackage{lscape}
\usepackage{tikz}
\usetikzlibrary{arrows,decorations.pathmorphing,decorations.pathreplacing}
\tikzset{black/.style={circle,fill=black,inner sep=3pt,outer sep=3pt},white/.style={circle,fill=white,draw=black,inner sep=3pt,outer sep=3pt}}
\usetikzlibrary{cd}
\usepackage{hyperref}
\numberwithin{equation}{section}
\newtheorem{theorem}{Theorem}[section]

\newtheorem{corollary}[theorem]{Corollary}
\newtheorem{lemma}[theorem]{Lemma}
\newtheorem{proposition}[theorem]{Proposition}
\theoremstyle{definition}
\newtheorem{definition}[theorem]{Definition}
\newtheorem{remark}[theorem]{Remark}
\newtheorem{example}[theorem]{Example}
\DeclareMathOperator{\tors}{\mathsf{tors}}

\DeclareMathOperator{\stors}{\mathsf{stors}}
\DeclareMathOperator{\fstors}{\mathsf{f-stors}}
\DeclareMathOperator{\tstr}{\mathsf{t-str}}

\begin{document}
\title[Intervals of $s$-torsion pairs]{Intervals of $s$-torsion pairs in extriangulated categories with negative first extensions}

\author{Takahide Adachi}
\address{T.~Adachi: Faculty of Global and Science Studies, Yamaguchi University, 1677-1 Yoshida, Yamaguchi 753-8541, Japan}
\email{tadachi@yamaguchi-u.ac.jp}
\thanks{T.~Adachi is supported by JSPS KAKENHI Grant Number JP20K14291.}

\author{Haruhisa Enomoto}
\address{H.~Enomoto: Graduate School of Mathematics, Nagoya University, Chikusa-ku, Nagoya 464-8602, Japan}
\email{m16009t@math.nagoya-u.ac.jp}
\thanks{H.~Enomoto is supported by JSPS KAKENHI Grant Number JP18J21556.}

\author{Mayu Tsukamoto}
\address{M.~Tsukamoto: Graduate school of Sciences and Technology for Innovation, Yamaguchi University, 1677-1 Yoshida, Yamaguchi 753-8512, Japan}
\email{tsukamot@yamaguchi-u.ac.jp}
\thanks{M.~Tsukamoto is supported by JSPS KAKENHI Grant Number JP19K14513.}


\begin{abstract}
As a general framework for the studies of $t$-structures on triangulated categories and torsion pairs in abelian categories, we introduce the notions of extriangulated categories with negative first extensions and $s$-torsion pairs.
We define a heart of an interval in the poset of $s$-torsion pairs, which naturally becomes an extriangulated category with a negative first extension.
This notion generalizes hearts of $t$-structures on triangulated categories and hearts of twin torsion pairs in abelian categories.
In this paper, we show that an interval in the poset of $s$-torsion pairs is bijectively associated with $s$-torsion pairs in the corresponding heart.
This bijection unifies two well-known bijections:
One is the bijection induced by HRS-tilt of $t$-structures on triangulated categories.
The other is Asai--Pfeifer's and Tattar's bijections for torsion pairs in an abelian category, which is related to $\tau$-tilting reduction and brick labeling.
\end{abstract}
\maketitle

\section{Introduction}
Be\u{\i}linson, Bernstein and Deligne \cite{BBD81} introduced the notion of \emph{$t$-structures} on triangulated categories to study perverse sheaves.
As one of remarkable properties of $t$-structures, various abelian categories inside triangulated categories can be realized as \emph{hearts} of $t$-structures.
Happel, Reiten and Smal{\o} \cite{HRS96} provided a construction of new $t$-structures through torsion pairs in the heart of a given $t$-structure.
This construction is called the \emph{HRS-tilt} and plays a crucial role in Bridgeland stability condition.
The HRS-tilt induces a close connection between $t$-structures and torsion pairs as follows.

\begin{theorem}\label{intro_thm1}\cite{HRS96,BR07,P06,W10}
Let $\mathcal{D}$ be a triangulated category with shift functor $\Sigma$.
Let $(\mathcal{U}, \mathcal{V})$ be a $t$-structure on  $\mathcal{D}$ and $\mathcal{H}:=\mathcal{U}\cap \Sigma\mathcal{V}$ the heart of $(\mathcal{U}, \mathcal{V})$.
Then there exists a poset isomorphism between the poset of $t$-structures $(\mathcal{U}', \mathcal{V}')$ on $\mathcal{D}$ satisfying $\Sigma\mathcal{U}\subseteq \mathcal{U}'\subseteq \mathcal{U}$ and the poset of torsion pairs in $\mathcal{H}$.
\end{theorem}

Recently, inspired by $\tau$-tilting theory \cite{AIR14}, the study of the poset structures of torsion pairs in abelian categories has been developed by various authors \cite{IRTT15, BCZ19, DIRRT, AP}.
Let $t_{1}:=(\mathcal{T}_{1}, \mathcal{F}_{1})$ and $t_{2}:=(\mathcal{T}_{2}, \mathcal{F}_{2})$ be torsion pairs in an abelian category $\mathcal{A}$.
We define $t_1 \leq t_2$ if it satisfies $\mathcal{T}_{1} \subseteq \mathcal{T}_{2}$, and in this case, $[t_{1},t_{2}]$ denotes the \emph{interval} in the poset of torsion pairs in $\mathcal{A}$ consisting of $t$ with $t_1 \leq t \leq t_2$.
We call the subcategory $\mathcal{H}_{[t_{1}, t_{2}]}:=\mathcal{T}_{2} \cap \mathcal{F}_{1}$ the \emph{heart} of $[t_{1}, t_{2}]$, which tells us ``the difference'' between $t_{1}$ and $t_{2}$.
As for this heart, the following isomorphism was established in \cite{AP} (when the heart is a wide subcategory) and \cite{T}.

\begin{theorem}\label{intro_thm2}\cite[Theorem A]{T}
Let $\mathcal{A}$ be an abelian category and $[t_{1},t_{2}]$ an interval in the poset of torsion pairs in $\mathcal{A}$.
Then there exists a poset isomorphism between $[t_{1},t_{2}]$ and the poset of torsion pairs in $\mathcal{H}_{[t_{1},t_{2}]}$.
\end{theorem}

This isomorphism originally appeared in the context of $\tau$-tilting reduction in \cite{J15}, and is a useful tool to study various subcategories of module categories, e.g. \cite{DIRRT, AP, ES}.

The aim of this paper is to show that two poset isomorphisms in Theorem \ref{intro_thm1} and Theorem \ref{intro_thm2} are consequences of a more general poset isomorphism in extriangulated categories.
We start with giving a common framework for studying $t$-structures on triangulated categories and torsion pairs in abelian categories.
In \cite{NP19}, Nakaoka and Palu introduced the notion of extriangulated categories as a simultaneous generalization of triangulated categories and exact categories.
Motivated by the fact that $t$-structures are exactly torsion pairs whose negative first extensions vanish (see Lemma \ref{lem_shift}), we introduce \emph{negative first extensions in extriangulated categories}, that is, an additive bifunctor $\mathbb{E}^{-1}$ satisfying a certain condition (see Definition \ref{def:negative_ext}).
We can naturally regard triangulated categories and exact categories as extriangulated categories with negative first extensions (see Example \ref{ex_neg}).
As a common generalization of $t$-structures and torsion pairs (in abelian categories), we introduce the notion of $s$-torsion pairs.
We call a pair $(\mathcal{T}, \mathcal{F})$ of subcategories an \emph{$s$-torsion pair} if it is a torsion pair (in the usual sense) and  $\mathbb{E}^{-1}(\mathcal{T}, \mathcal{F})=0$ holds.
Since the set of $s$-torsion pairs becomes a partially ordered set by inclusions, we can define the \emph{hearts} of intervals as with torsion pairs in abelian categories.
Moreover, each heart can be naturally regarded as an extriangulated category with a negative first extension.
In this setting, the following theorem is a main result of this paper.

\begin{theorem}[Theorem \ref{mainthm}] \label{intro_thm}
Let $\mathcal{C}$ be an extriangulated category with a negative first extension.
Let $[t_{1}, t_{2}]$ be an interval in the poset of $s$-torsion pairs in $\mathcal{C}$ and $\mathcal{H}_{[t_{1}, t_{2}]}$ its heart.
Then there exists a poset isomorphism between $[t_{1}, t_{2}]$ and the poset of $s$-torsion pairs in $\mathcal{H}_{[t_{1}, t_{2}]}$.
\end{theorem}

Regarding triangulated categories and abelian categories as extriangulated categories with negative first extensions, we can immediately recover Theorem \ref{intro_thm1} and Theorem \ref{intro_thm2} (see Corollary \ref{recover-tstr} and Corollary \ref{recover-tors} respectively).

\section{Extriangulated categories with negative first extensions}
Throughout this paper, we assume that every category is skeletally small, that is, the isomorphism classes of objects form a set.
In addition, all subcategories are assumed to be full and closed under isomorphisms.

In this section, we collect terminologies and basic properties of extriangulated categories which we need later. Moreover, we introduce a  negative first extension structure of an extriangulated category.
We omit the precise definition of extriangulated categories.
For details, we refer to \cite{NP19}.

An extriangulated category $\mathcal{C}=(\mathcal{C},\mathbb{E},\mathfrak{s})$ consists of the following data which satisfy certain axioms (see \cite[Definition 2.12]{NP19}).
\begin{itemize}
\item $\mathcal{C}$ is an additive category.
\item $\mathbb{E}:\mathcal{C}^{\mathrm{op}}\times \mathcal{C}\to \mathcal{A}b$ is an additive bifunctor, where $\mathcal{A}b$ denotes the category of abelian groups.
\item $\mathfrak{s}$ is a correspondence which associates an equivalence class $[A\rightarrow B\rightarrow C]$ of complexes in $\mathcal{C}$ to each $\delta\in \mathbb{E}(C,A)$.
\end{itemize}
Here two complexes $A\xrightarrow{f}B\xrightarrow{g}C$ and $A\xrightarrow{f'}B'\xrightarrow{g'}C$ in $\mathcal{C}$ are \emph{equivalent} if there is an isomorphism $b:B\to B'$ such that the diagram
\[
\begin{tikzcd}
A \rar["f"] \dar[equal] & B \rar["g"] \dar["b"', "\cong"] & C \dar[equal] \\
A \rar["f'"] & B' \rar["g'"] & C
\end{tikzcd}
\]
is commutative, and let $[A\xrightarrow{f}B\xrightarrow{g}C]$ denote the equivalence class of $A\xrightarrow{f}B\xrightarrow{g}C$.

A complex $A\xrightarrow{f}B\xrightarrow{g}C$ in $\mathcal{C}$ is called an \emph{$\mathfrak{s}$-conflation} if there exists $\delta\in \mathbb{E}(C,A)$ such that $\mathfrak{s}(\delta)=[A\xrightarrow{f}B\xrightarrow{g}C]$.
We often write the $\mathfrak{s}$-conflation as $A\xrightarrow{f}B\xrightarrow{g}C\overset{\delta}{\dashrightarrow}$.
Let $\delta\in \mathbb{E}(C,A)$.
By Yoneda's lemma, we have two natural transformations $\delta_{\sharp}:\mathcal{C}(-,C)\to \mathbb{E}(-,A)$ and $\delta^{\sharp}:\mathcal{C}(A,-)\to\mathbb{E}(C,-)$, that is, for each $W\in\mathcal{C}$,
\begin{align}
&(\delta_{\sharp})_{W}: \mathcal{C}(W,C)\to \mathbb{E}(W,A)\ \ (\varphi\mapsto\mathbb{E}(\varphi,A)(\delta)),\notag\\
&(\delta^{\sharp})_{W}: \mathcal{C}(A,W)\to \mathbb{E}(C,W)\ \ (\varphi\mapsto\mathbb{E}(C,\varphi)(\delta)).\notag
\end{align}
Any $\mathfrak{s}$-conflation induces two long exact sequences in $\mathcal{A}b$.

\begin{proposition}\cite[Corollary 3.12]{NP19}\label{prop_longex}
Let $\mathcal{C}$ be an extriangulated category and
$A\xrightarrow{f}B\xrightarrow{g}C\overset{\delta}{\dashrightarrow}$ an $\mathfrak{s}$-conflation.
Then, for each $W\in\mathcal{C}$, two sequences
\[
\begin{tikzcd}[row sep = 0]
\mathcal{C}(W,A) \rar["{\mathcal{C}(W,f)}"] & \mathcal{C}(W,B) \rar["{\mathcal{C}(W,g)}"] & \mathcal{C}(W,C)\rar["{(\delta_{\sharp})_{W}}"]  & \mathbb{E}(W,A) \rar["{\mathbb{E}(W,f)}"] & \mathbb{E}(W,B) \rar["{\mathbb{E}(W,g)}"] & \mathbb{E}(W,C),\\
\mathcal{C}(C,W) \rar["{\mathcal{C}(g,W)}"] & \mathcal{C}(B,W) \rar["{\mathcal{C}(f,W)}"] & \mathcal{C}(A,W)\rar["{(\delta^{\sharp})_{W}}"] & \mathbb{E}(C,W) \rar["{\mathbb{E}(g,W)}"] & \mathbb{E}(B,W) \rar["{\mathbb{E}(f,W)}"] & \mathbb{E}(A,W)
\end{tikzcd}
\]
are exact.
\end{proposition}

We define extension-closed subcategories of extriangulated categories.

\begin{definition}
Let $\mathcal{C}$ be an extriangulated category.
\begin{itemize}
\item[(1)] For two collections $\mathcal{X}$ and $\mathcal{Y}$ of objects in $\mathcal{C}$, let $\mathcal{X}\ast\mathcal{Y}$ denote the subcategory of $\mathcal{C}$ consisting of $M\in\mathcal{C}$ which admits an $\mathfrak{s}$-conflation $X\rightarrow M\rightarrow Y\dashrightarrow$ with $X\in \mathcal{X}$ and $Y\in\mathcal{Y}$.
\item[(2)] We say that a subcategory $\mathcal{C}'$ of $\mathcal{C}$ is an \emph{extension-closed subcategory} if $\mathcal{C}'\ast \mathcal{C}'\subseteq \mathcal{C}'$.
\end{itemize}
\end{definition}

For three collections $\mathcal{X},\mathcal{Y},\mathcal{Z}$ of objects in $\mathcal{C}$, we can easily check $(\mathcal{X}\ast \mathcal{Y})\ast \mathcal{Z}=\mathcal{X}\ast(\mathcal{Y}\ast \mathcal{Z})$ by the axioms (ET4) and (ET4)$^{\mathrm{op}}$ in \cite[Definition 2.12]{NP19}, that is, the operation $\ast$ is associative.
If $\mathcal{Y}$ contains a zero object in $\mathcal{C}$, then we have $\mathcal{X}\subseteq \mathcal{X}\ast \mathcal{Y}$ and $\mathcal{X}\subseteq \mathcal{Y}\ast \mathcal{X}$.
Indeed, for each $X\in \mathcal{X}$, there exist two $\mathfrak{s}$-conflations
\[
\begin{tikzcd}[row sep = 0]
X \rar["{\mathsf{id}_{X}}"] & X \rar & 0 \rar[dashrightarrow] & , \\
0 \rar & X \rar["\mathsf{id}_{X}"] & X \rar[dashed] &\
\end{tikzcd}
\]
by the additivity of $\mathfrak{s}$ (see \cite[Definition 2.10]{NP19}).
Thus $X\in\mathcal{X}\ast\mathcal{Y}$ follows from the first $\mathfrak{s}$-conflation and $X\in \mathcal{Y}\ast\mathcal{X}$ follows from the second one.
By \cite[Remark 2.18]{NP19}, extension-closed subcategories of extriangulated categories are naturally extriangulated categories.

Now we introduce a negative first extension structure on an extriangulated category.

\begin{definition}\label{def:negative_ext}
Let $\mathcal{C}$ be an extriangulated category.
A \emph{negative first extension structure} on $\mathcal{C}$ consists of the following data:
\begin{itemize}
\setlength{\itemindent}{20pt}
\item[(NE1)] $\mathbb{E}^{-1}:\mathcal{C}^{\mathrm{op}}\times \mathcal{C}\to \mathcal{A}b$ is an additive bifunctor.
\item[(NE2)] For each $\delta\in \mathbb{E}(C,A)$, there exist two natural transformations
\begin{align}
&\delta_{\sharp}^{-1}: \mathbb{E}^{-1}(-,C)\to \mathcal{C}(-,A),\notag\\
&\delta^{\sharp}_{-1}: \mathbb{E}^{-1}(A,-)\to \mathcal{C}(C,-)\notag
\end{align}
such that for each $\mathfrak{s}$-conflation
$A\xrightarrow{f}B\xrightarrow{g}C\overset{\delta}{\dashrightarrow}$ and each $W\in \mathcal{C}$, two sequences
\[
\begin{tikzcd}[row sep = 0, column sep=1.4cm]
\mathbb{E}^{-1}(W,A) \rar["{\mathbb{E}^{-1}(W,f)}"] & \mathbb{E}^{-1}(W,B) \rar["{\mathbb{E}^{-1}(W,g)}"] & \mathbb{E}^{-1}(W,C) \rar["{(\delta^{-1}_{\sharp})_{W}}"]&\mathcal{C}(W,A) \rar["{\mathcal{C}(W,f)}"]&\mathcal{C}(W,B),\\
\mathbb{E}^{-1}(C,W) \rar["{\mathbb{E}^{-1}(g,W)}"] & \mathbb{E}^{-1}(B,W) \rar["{\mathbb{E}^{-1}(f,W)}"] & \mathbb{E}^{-1}(A,W) \rar["{(\delta_{-1}^{\sharp})_{W}}"]&\mathcal{C}(C,W) \rar["{\mathcal{C}(g,W)}"]&\mathcal{C}(B,W)
\end{tikzcd}
\]
are exact.
\end{itemize}
Then we call $\mathcal{C}=(\mathcal{C},\mathbb{E},\mathfrak{s},\mathbb{E}^{-1})$ an \emph{extriangulated category with a negative first extension}.
\end{definition}

Typical examples of extriangulated categories are triangulated categories and exact categories (see \cite[Example 2.13]{NP19}).
We show that both categories naturally admit negative first extension structures.
In the rest of this paper, unless otherwise stated, we always regard these categories as extriangulated categories with negative first extensions defined below.

\begin{example}\label{ex_neg}
\begin{itemize}
\item[(1)] A triangulated category $\mathcal{D}$ becomes an extriangulated category with a negative first extension by the following data.
\begin{itemize}
\item[$\bullet$] $\mathbb{E}(C,A):=\mathcal{D}(C,\Sigma A)$ for all $A,C\in \mathcal{D}$, where $\Sigma$ is a shift functor of $\mathcal{D}$.
\item[$\bullet$] For $\delta\in \mathbb{E}(C,A)$, we take a triangle $A\xrightarrow{f}B\xrightarrow{g}C\xrightarrow{\delta}\Sigma A$.
Then we define $\mathfrak{s}(\delta):=[A\xrightarrow{f}B\xrightarrow{g}C]$.
\item[$\bullet$] $\mathbb{E}^{-1}(C,A):=\mathcal{D}(C,\Sigma^{-1}A)$ for all $A,C\in \mathcal{D}$.
\item[$\bullet$] For an $\mathfrak{s}$-conflation $A\xrightarrow{f}B\xrightarrow{g}C\overset{\delta}\dashrightarrow$, we define two natural transformations $\delta_{\sharp}^{-1}$ and $\delta_{-1}^{\sharp}$ as follows: for $W\in \mathcal{D}$,
\begin{align}
&(\delta_{\sharp}^{-1})_{W}:\mathbb{E}^{-1}(W,C)=\mathcal{D}(W,\Sigma^{-1}C)\xrightarrow{\mathcal{D}(W,\Sigma^{-1}\delta)}\mathcal{D}(W,A),\notag\\
&(\delta_{-1}^{\sharp})_{W}:\mathbb{E}^{-1}(A,W)=\mathcal{D}(A,\Sigma^{-1}W)\cong \mathcal{D}(\Sigma A,W)\xrightarrow{\mathcal{D}(\delta,W)}\mathcal{D}(C,W).\notag
\end{align}
\end{itemize}
Note that the negative sign in the axiom (TR2) of triangulated categories does not affect the exactness of (NE2) in Definition \ref{def:negative_ext}.
\item[(2)] An exact category $\mathcal{E}$ becomes an extriangulated category with a negative first extension by the following data.
\begin{itemize}
\item[$\bullet$] $\mathbb{E}(C,A)$ is the set of isomorphism classes of conflations in $\mathcal{E}$ of the form $A \rightarrowtail B \twoheadrightarrow C$ for $A,C \in \mathcal{E}$.
\item[$\bullet$] $\mathfrak{s}$ is the identity.
\item[$\bullet$] $\mathbb{E}^{-1}(C,A)=0$ for all $A,C\in \mathcal{E}$.
\item[$\bullet$] For each $W\in \mathcal{E}$, the maps $(\delta_{\sharp}^{-1})_{W}$ and $(\delta_{-1}^{\sharp})_{W}$ are zero.
\end{itemize}
Note that the above $\mathbb{E}^{-1}$ defines a negative first extension structure since every inflation is a monomorphism and every deflation is an epimorphism. Actually, the converse holds, see Proposition \ref{prop_exactcat}.
\item[(3)] Let $\mathcal{C}$ be an extriangulated category with a negative first extension and let $\mathcal{C}'$ be an extension-closed subcategory of $\mathcal{C}$. Then by restricting the extriangulated structure and the negative first extension structure to $\mathcal{C}'$, we can regard $\mathcal{C}'$ as an extriangulated category with a negative first extension.
\end{itemize}
\end{example}

Remark that negative first extension structures are not uniquely determined by given extriangulated categories.
\begin{example}\label{ex_nakayama}
Let $k$ be an algebraically closed field.
Consider the stable category $\mathcal{D}:=\operatorname{\underline{\mathsf{mod}}}\Lambda$ of a self-injective Nakayama $k$-algebra $\Lambda$ with three simple modules and the Loewy length three.
Then the Auslander-Reiten quiver of $\mathcal{D}$ is as follows, where two $\substack{1}$'s are identified.

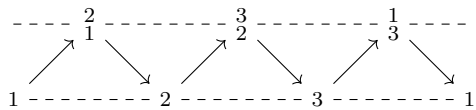
\[
\begin{tikzpicture}
\node (1) at (0,0) {$\substack{1}$};
\node (2) at (2,0) {$\substack{2}$};
\node (3) at (4,0) {$\substack{3}$};
\node (11) at (6,0) {$\substack{1}$};
\node (21) at (1,1) {$\substack{2\\1}$};
\node (32) at (3,1) {$\substack{3\\2}$};
\node (13) at (5,1) {$\substack{1\\3}$};

\draw[->] (1) -- (21);
\draw[->] (21) -- (2);
\draw[->] (2) -- (32);
\draw[->] (32) -- (3);
\draw[->] (3) -- (13);
\draw[->] (13) -- (11);

\draw[dashed] (1) -- (2) -- (3) -- (11);
\draw[dashed] (0,1) -- (21) -- (32) -- (13) -- (6,1);
\end{tikzpicture}
\]
Since the subcategory $\mathcal{A} := \operatorname{\mathsf{add}}\{\substack{1},\substack{2\\1}, \substack{2}\}$ is clearly equivalent to the category of finite-dimensional representations of an $A_{2}$ quiver, it is abelian.
Thus $\mathcal{A}$ becomes an extriangulated category with a negative first extension $\mathbb{E}_{1}^{-1}:=0$ by Example \ref{ex_neg}(2).
On the other hand, since $\mathcal{A}$ is extension-closed in $\mathcal{D}$, it becomes an induced extriangulated category with a  negative first extension $\mathbb{E}_{2}^{-1}(-, -):=\mathcal{D}(-, \Sigma^{-1} -)$ by Example \ref{ex_neg}(3).
We can check that extriangulated category structures coincide with each other, but negative first extension structures do not.
Indeed, $\mathbb{E}_{2}^{-1}(\substack{2},\substack{1})=\mathcal{D}(\substack{2}, \Sigma^{-1}(\substack{1}))=\mathcal{D}(\substack{2}, \substack{3\\2})\neq 0$ holds.
We refer the reader to Example \ref{ex_fingldim} and Remark \ref{rem_periodic} for more examples of this kind.
\end{example}

Using negative first extension structures, we give a characterization of extriangulated categories to be exact categories.

\begin{proposition}\label{prop_exactcat}
Let $(\mathcal{C},\mathbb{E},\mathfrak{s})$ be an extriangulated category.
Then the following statements are equivalent.
\begin{itemize}
\item[(1)] $\mathbb{E}^{-1} = 0$ defines a negative first extension structure on $(\mathcal{C},\mathbb{E},\mathfrak{s})$.
\item[(2)] The class of $\mathfrak{s}$-conflations defines the structure of an exact category on $\mathcal{C}$.
\end{itemize}
\end{proposition}

\begin{proof}
In \cite[Corollary 3.18]{NP19}, it is shown that (2) is equivalent to the condition that for every $\mathfrak{s}$-conflation $A \xrightarrow{x} B \xrightarrow{y} C\dashrightarrow$, we obtain that $x$ is a monomorphism and $y$ is an epimorphism.

(1)$\Rightarrow$(2):
The long exact sequences in (NE2) clearly imply that $x$ is a monomorphism and $y$ is an epimorphism in $\mathcal{C}$. Thus (2) holds.

(2)$\Rightarrow$(1):
The sequences in (NE2) are exact for $\mathbb{E}^{-1}:= 0$ since every inflation is a monomorphism and every deflation is an epimorphism.
\end{proof}

As a byproduct, we give the following criterion for an extension-closed subcategory of a triangulated category to be an exact category, which is proved in \cite[Proposition 2.5]{J20} and \cite[Theorem]{D} in different ways.

\begin{corollary}
Let $\mathcal{D}$ be a triangulated category with shift functor $\Sigma$ and let $\mathcal{C}$ be an extension-closed subcategory of $\mathcal{D}$.
Assume that $\mathcal{D}(\mathcal{C},\Sigma^{-1}\mathcal{C}) = 0$ holds. Then $\mathcal{C}$ has a structure of an exact category, whose conflations are precisely triangles of $\mathcal{D}$ with the first three terms in $\mathcal{C}$.
\end{corollary}

\begin{proof}
Since $\mathcal{C}$ is an extension-closed subcategory of a triangulated category $\mathcal{D}$, we regard $\mathcal{C}$ as an extriangulated category with a negative first extension  $\mathbb{E}^{-1}(-,-)|_{\mathcal{C}}:=\mathcal{D}(-, \Sigma^{-1}-)|_{\mathcal{C}}$ by Example \ref{ex_neg}(1) and (3).
On the other hand, the assumption $\mathcal{D}(\mathcal{C},\Sigma^{-1}\mathcal{C}) = 0$ implies $\mathbb{E}^{-1}(-,-)|_{\mathcal{C}}=0$.
Thus the assertion follows from Proposition \ref{prop_exactcat}.
\end{proof}

\section{$s$-torsion pairs in extriangulated categories with negative first extensions}
In this section, we study $s$-torsion pairs in an extriangulated category with a negative first extension.
In the following, let $\mathcal{C}=(\mathcal{C},\mathbb{E},\mathfrak{s},\mathbb{E}^{-1})$ be an extriangulated category with a negative first extension.
For a collection $\mathcal{X}$ of objects in $\mathcal{C}$, we define $\mathcal{X}^\perp := \{ C \in \mathcal{C} \mid \mathcal{C}(\mathcal{X},C) = 0 \}$ and $^\perp \mathcal{X} := \{ C \in \mathcal{C} \mid \mathcal{C}(C,\mathcal{X}) = 0 \}$.

\subsection{$s$-torsion pairs}
We introduce the notion of $s$-torsion pairs in an extriangulated category with a negative first extension.

\begin{definition}
Let $\mathcal{C}$ be an extriangulated category with a negative first extension.
We call a pair $(\mathcal{T},\mathcal{F})$ of subcategories of $\mathcal{C}$ an \emph{$s$-torsion pair} in $\mathcal{C}$ if it satisfies the following three conditions.
\begin{itemize}
\setlength{\itemindent}{20pt}
\item[(STP1)] $\mathcal{C}=\mathcal{T} \ast \mathcal{F}$.
\item[(STP2)] $\mathcal{C}(\mathcal{T}, \mathcal{F})=0$.
\item[(STP3)] $\mathbb{E}^{-1}(\mathcal{T}, \mathcal{F})=0$.
\end{itemize}
In this case, $\mathcal{T}$ (respectively, $\mathcal{F}$) is called a \emph{torsion class} (respectively, \emph{torsion-free class}) in $\mathcal{C}$.
\end{definition}

Let $\stors \mathcal{C}$ denote the set of $s$-torsion pairs.
We write $(\mathcal{T}_{1},\mathcal{F}_{1})\le (\mathcal{T}_{2},\mathcal{F}_{2})$ if $\mathcal{T}_{1}\subseteq\mathcal{T}_{2}$.
Then $(\stors \mathcal{C}, \le)$ clearly becomes a partially ordered set.

By the following proposition, a torsion-free class (respectively, torsion class) is uniquely determined by a torsion class (respectively, torsion-free class), that is, if $(\mathcal{T},\mathcal{F})$ and $(\mathcal{T},\mathcal{F}')$ are $s$-torsion pairs, then we have $\mathcal{F}=\mathcal{F}'$.

\begin{proposition}\label{prop_orth}
Let $(\mathcal{T}, \mathcal{F})$ be an $s$-torsion pair in $\mathcal{C}$.
Then the following statements hold.
\begin{itemize}
\item[(1)] $\mathcal{T}^{\perp}=\mathcal{F}$.
\item[(2)] ${}^{\perp}\mathcal{F}=\mathcal{T}$.
\end{itemize}
In particular, $\mathcal{T}$ and $\mathcal{F}$ are extension-closed subcategories which are closed under direct summands.
\end{proposition}

\begin{proof}
We only prove (1); the proof of (2) is similar.
By (STP2), $\mathcal{T}^{\perp}\supseteq \mathcal{F}$ holds.
We show the converse inclusion.
Let $C \in \mathcal{T}^{\perp}$.
By (STP1), there exists an $\mathfrak{s}$-conflation
\[
\begin{tikzcd}
T\rar["f"] & C \rar["g"] & F \rar[dashrightarrow] & \
\end{tikzcd}
\]
such that $T \in \mathcal{T}$ and $F \in \mathcal{F}$.
Applying $\mathcal{C}(T,-)$ to the $\mathfrak{s}$-conflation gives an exact sequence
\[
\begin{tikzcd}
\mathbb{E}^{-1}(T, F) \rar & \mathcal{C}(T, T) \rar & \mathcal{C}(T,C).
\end{tikzcd}
\]
Since the left-hand side and the right-hand side vanish by (STP3) and $C \in \mathcal{T}^{\perp}$ respectively, we obtain $\mathcal{C}(T, T)=0$, and hence $T=0$.
This implies that the natural transformation $\mathcal{C}(-,g): \mathcal{C}(-,C) \to \mathcal{C}(-,F)$ is an isomorphism by Proposition \ref{prop_longex}.
Thus $C \cong F \in \mathcal{F}$ by Yoneda's lemma.
\end{proof}

As an immediate consequence, for $s$-torsion pairs $(\mathcal{T}_{1},\mathcal{F}_{1})$ and $(\mathcal{T}_{2},\mathcal{T}_{2})$, we obtain that $\mathcal{T}_{1}\subseteq \mathcal{T}_{2}$ if and only if $\mathcal{F}_{1}\supseteq \mathcal{F}_{2}$.

The term ``$s$'' in an $s$-torsion pair stands for ``shift-closed'' by the mean of the following lemma.

\begin{lemma}\label{lem_shift}
Let $\mathcal{D}$ be a triangulated category (regarded as the extriangulated category with the negative first extension by Example \ref{ex_neg}(1)).
Let $(\mathcal{T},\mathcal{F})$ be a pair of subcategories of $\mathcal{D}$ satisfying the conditions \textnormal{(STP1)} and \textnormal{(STP2)}.
Then the following statements are equivalent.
\begin{itemize}
\item[(1)] $(\mathcal{T},\mathcal{F})$ satisfies the condition \textnormal{(STP3)}.
\item[(2)] $\mathcal{T}$ is closed under a positive shift, that is, $\Sigma\mathcal{T}\subseteq \mathcal{T}$.
\item[(3)] $\mathcal{F}$ is closed under a negative shift, that is, $\mathcal{F}\subseteq\Sigma\mathcal{F}$.
\end{itemize}
\end{lemma}

\begin{proof}
We only prove (1)$\Leftrightarrow$(2); the proof of (1)$\Leftrightarrow$(3) is similar.
By Example \ref{ex_neg}(1), we have $\mathbb{E}^{-1}(\mathcal{T},\mathcal{F})\cong \mathcal{D}(\Sigma \mathcal{T},\mathcal{F})$.
Hence (STP3) holds if and only if $\Sigma\mathcal{T}\subseteq {}^{\perp}\mathcal{F}$.
The assertion follows from Proposition \ref{prop_orth}(2).
\end{proof}

The following examples show that $s$-torsion pairs are a common generalization of $t$-structures on triangulated categories and torsion pairs in exact categories.

\begin{example}\label{ex_tstr}
Let $\mathcal{D}$ be a triangulated category. A pair $(\mathcal{U},\mathcal{V})$ of subcategories of $\mathcal{D}$ is called a \emph{$t$-structure} on $\mathcal{D}$ if it satisfies the following three conditions.
\begin{itemize}
\item[$\bullet$] $\mathcal{D}=\mathcal{U}\ast\mathcal{V}$, that is, for each $D\in\mathcal{D}$, there exists a triangle $U\to D\to V\to \Sigma U$ such that $U\in \mathcal{U}$ and $V\in\mathcal{V}$.
\item[$\bullet$] $\mathcal{D}(\mathcal{U},\mathcal{V})=0$.
\item[$\bullet$] $\mathcal{U}$ is closed under a positive shift.
\end{itemize}
It is a well-known result \cite{BBD81} that the heart $\mathcal{U}\cap \Sigma\mathcal{V}$ of a $t$-structure $(\mathcal{U},\mathcal{V})$ is always an abelian category.
By regarding $\mathcal{D}$ as the extriangulated category with the negative first extension (see Example \ref{ex_neg}(1)), it follows from Lemma \ref{lem_shift} that $t$-structures on  $\mathcal{D}$ are exactly $s$-torsion pairs in $\mathcal{D}$.
\end{example}

\begin{example}\label{ex_extor}
Let $\mathcal{E}$ be an exact category. A pair $(\mathcal{T},\mathcal{F})$ of subcategories of $\mathcal{E}$ is called a \emph{torsion pair} in $\mathcal{E}$ if it satisfies the following two conditions.
\begin{itemize}
\item[$\bullet$] $\mathcal{E}=\mathcal{T}\ast\mathcal{F}$, that is, for each $E\in \mathcal{E}$, there exists a conflation $0\to T\to E\to F\to 0$ such that $T\in\mathcal{T}$ and $F\in\mathcal{F}$.
\item[$\bullet$] $\mathcal{E}(\mathcal{T},\mathcal{F})=0$.
\end{itemize}
Let $\tors\mathcal{E}$ denote the poset of torsion pairs in $\mathcal{E}$, where we define $(\mathcal{T}_{1},\mathcal{F}_{1}) \leq (\mathcal{T}_{2},\mathcal{F}_{2})$ if $\mathcal{T}_{1} \subseteq \mathcal{T}_{2}$.
By regarding $\mathcal{E}$ as the extriangulated category with the negative first extension (see Example \ref{ex_neg}(2)), it follows from $\mathbb{E}^{-1}=0$ that torsion pairs in the exact category $\mathcal{E}$ are exactly $s$-torsion pairs  in $\mathcal{E}$, that is, we have $\tors\mathcal{E} = \stors\mathcal{E}$.
\end{example}

Taking negative first extension structures different from Example \ref{ex_neg}(2), we give an example which satisfies (STP1) and (STP2) but does not satisfy (STP3).

\begin{example}
Let $\Lambda$ and $\mathcal{A}$ be in Example \ref{ex_nakayama}.
Due to Example \ref{ex_nakayama}, we regard $\mathcal{A}$ as the extriangulated category with the negative first extension $\mathbb{E}_{2}^{-1}$.
Since $\mathcal{A}$ is an abelian category, a pair of subcategories $(\mathcal{T},\mathcal{F})$ satisfies (STP1) and (STP2) if and only if it is a (usual) torsion pair in the abelian category $\mathcal{A}$.
Thus, $(\operatorname{\mathsf{add}}\{\substack{2\\1},\substack{2}\}, \operatorname{\mathsf{add}}\{\substack{1}\})$ satisfies (STP1) and (STP2).
On the other hand, since $\mathbb{E}_{2}^{-1}(\substack{2},\substack{1}) \neq 0$ holds, this pair does not satisfy (STP3).
\end{example}

We show that the $\mathfrak{s}$-conflation in (STP1) is unique up to isomorphism.
As a result, we have right and left adjoint functors of inclusion functors, which are generalizations of torsion radicals for the usual torsion pairs and truncation functors for $t$-structures.

\begin{proposition}\label{prop_canseq}
Let $(\mathcal{T},\mathcal{F})$ be an $s$-torsion pair in $\mathcal{C}$. For each $C \in \mathcal{C}$, there uniquely exists an $\mathfrak{s}$-conflation
\begin{equation}\label{canseq}
\begin{tikzcd}
T_{C}\rar["\iota_{C}"] & C \rar["{\pi^{C}}"] & F^{C} \rar["\delta", dashrightarrow] & \
\end{tikzcd}
\end{equation}
with $T_{C} \in \mathcal{T}$ and $F^{C} \in \mathcal{F}$ (up to isomorphism of $\mathfrak{s}$-conflations).
Moreover the assignment $C \mapsto T_{C}$ and $C \mapsto F^{C}$ induce functors $T_{(-)} : \mathcal{C} \to \mathcal{T}$ and $F^{(-)} : \mathcal{C} \to \mathcal{F}$ which are right adjoint and left adjoint of the inclusion functors $\mathcal{T} \rightarrow \mathcal{C}$ and $\mathcal{F} \rightarrow \mathcal{C}$ respectively.
\end{proposition}

\begin{proof}
Let $C \in \mathcal{C}$.
By (STP1), we take an $\mathfrak{s}$-conflation $T \xrightarrow{\iota} C \xrightarrow{\pi} F \dashrightarrow$ such that $T \in \mathcal{T}$ and $F\in \mathcal{F}$.
By (STP2) and (STP3), we have an isomorphism
\[
\begin{tikzcd}
0=\mathbb{E}^{-1}(T', F) \rar & \mathcal{C}(T', T) \rar["{\mathcal{C}(T', \iota)}"] &  \mathcal{C}(T', C) \rar & \mathcal{C}(T',F)=0
\end{tikzcd}
\]
for each $T' \in \mathcal{T}$.
This implies that the natural transformation $\mathcal{C}(-,\iota) :\mathcal{T}(-,T) \to \mathcal{C}(-,C)|_{\mathcal{T}}$ between two functors on $\mathcal{T}$ is an isomorphism.

We show the uniqueness of \eqref{canseq}.
We take two $\mathfrak{s}$-conflations
\[
\begin{tikzcd}[row sep = 0]
T_{1}\rar["\iota_{1}"] & C \rar["\pi_{1}"] & F_{1}\rar["\delta_1", dashrightarrow] &, \ \\
T_{2}\rar["\iota_{2}"] & C \rar["\pi_{2}"] & F_{2}\rar["\delta_2", dashrightarrow] & \
\end{tikzcd}
\]
with $T_{1}, T_{2} \in \mathcal{T}$ and $F_{1}, F_{2} \in \mathcal{F}$.
Since $\mathcal{C}(-,\iota_{i}) : \mathcal{T}(-,T_{i}) \to \mathcal{C}(-,C)|_{\mathcal{T}}$ is a natural isomorphism for $i=1,2$, there is an isomorphism $\varphi : T_{1} \rightarrow T_{2}$ satisfying $\iota_{2}\varphi = \iota_{1}$. By the axiom (ET3) in \cite[Definition 2.12]{NP19}, we obtain a morphism of $\mathfrak{s}$-conflations
\[
\begin{tikzcd}
T_1 \rar["\iota_{1}"] \dar["\varphi"', "\cong"] & C \dar[equal] \rar["\pi_{1}"] & F_{1} \rar[dashed, "{\delta_{1}}"] \dar["\psi"] & \ \\
T_2 \rar["\iota_{2}"] & C \rar["\pi_{2}"] & F_{2} \rar[dashed, "{\delta_{2}}"] &\rlap{.} \
\end{tikzcd}
\]
Then \cite[Corollary 3.6]{NP19} implies that $\psi$ is an isomorphism.
Note that this isomorphism also follows from a natural isomorphism $\mathcal{C}(\pi_{i},-) :\mathcal{F}(F_{i},-) \to \mathcal{C}(C,-)|_{\mathcal{F}}$ of functors on $\mathcal{F}$ for $i=1,2$.
Thus these two $\mathfrak{s}$-conflations are isomorphic to each other.

For each $C \in \mathcal{C}$, fix an $\mathfrak{s}$-conflation (\ref{canseq}).
Since we have a natural isomorphism $\mathcal{C}(-,\iota_{C}) :\mathcal{T}(-,T_{C}) \to \mathcal{C}(-,C)|_{\mathcal{T}}$ on $\mathcal{T}$ for each $C$, Yoneda's lemma immediately implies that $C \mapsto T_{C}$ gives a functor $T_{(-)} : \mathcal{C} \to \mathcal{T}$. Moreover, it is clear from this natural isomorphism that the functor $T_{(-)}$ is a right adjoint functor of the inclusion functor $\mathcal{T} \rightarrow \mathcal{C}$.
Similarly, we can show that the functor $F^{(-)}: \mathcal{C}\to \mathcal{F}$ is a left adjoint functor of the inclusion functor $\mathcal{F} \to \mathcal{C}$.
\end{proof}

\subsection{Intervals and hearts of $s$-torsion pairs}
In this subsection, we give a common generalization of the isomorphism (Theorem \ref{intro_thm1}) induced by HRS-tilt of $t$-structures on a triangulated category and the isomorphism (Theorem \ref{intro_thm2}) via torsion pairs in an abelian category.

The following notion plays an important role in this paper.
\begin{definition}
Let $\mathcal{C}$ be an extriangulated category with a negative first extension.
For $i=1,2$, let $t_{i}:=(\mathcal{T}_{i},\mathcal{F}_{i})$ $\in \stors \mathcal{C}$ with $t_{1}\le t_{2}$.
Then we call the subposet
\begin{align}
\stors [t_{1}, t_{2}]:=\{t=(\mathcal{T}, \mathcal{F})\in \stors \mathcal{C} \mid t_{1} \le t \le t_{2}\} \subseteq \stors\mathcal{C}\notag
\end{align}
an \emph{interval} in $\stors \mathcal{C}$ and the subcategory $\mathcal{H}_{[t_{1},t_{2}]} := \mathcal{T}_{2} \cap \mathcal{F}_{1}\subseteq \mathcal{C}$ the \emph{heart} of the interval $\stors [t_{1},t_{2}]$.
Since $\mathcal{H}_{[t_{1},t_{2}]}$ is extension-closed, we can regard $\mathcal{H}_{[t_{1},t_{2}]}$ as the extriangulated category with the negative first extension (see Example \ref{ex_neg}(3)).
\end{definition}

By Example \ref{ex_tstr}, $t$-structures on a triangulated category $\mathcal{D}$ are exactly $s$-torsion pairs in the extriangulated category $\mathcal{D}$.
We can easily check that the heart of a $t$-structure $(\mathcal{U},\mathcal{V})$ on $\mathcal{D}$ coincides with the heart of the interval $\stors[(\Sigma\mathcal{U},\Sigma\mathcal{V}),(\mathcal{U},\mathcal{V})]$.
For an abelian category $\mathcal{A}$, the notion of hearts of intervals in $\tors\mathcal{A}$ was used implicitly in \cite{J15} to study $\tau$-tilting reduction, and in \cite{DIRRT,AP} to study the lattice structure and the brick labeling of $\tors\mathcal{A}$.
The terminology \emph{heart} was introduced in \cite{T}.
Moreover, this notion plays a crucial role in the study of subcategories of abelian categories closed under images, cokernels and extensions in \cite{ES}.

Now we state a main result of this paper.

\begin{theorem}\label{mainthm}
Let $\mathcal{C}$ be an extriangulated category with a negative first extension.
For $i=1,2$, let $t_{i}:=(\mathcal{T}_{i},\mathcal{F}_{i})\in \stors\mathcal{C}$ with $t_{1}\le t_{2}$.
Then there exist mutually inverse isomorphisms of posets
\[
\begin{tikzcd}
\stors {[t_{1}, t_{2}]} \rar[shift left, "\Phi"] & \stors\mathcal{H}_{[t_{1}, t_{2}]}, \lar[shift left, "\Psi"]
\end{tikzcd}
\]
where $\Phi(\mathcal{T},\mathcal{F}):=(\mathcal{T}\cap\mathcal{F}_{1}, \mathcal{T}_{2} \cap \mathcal{F})$ and $\Psi(\mathcal{X},\mathcal{Y}):=(\mathcal{T}_{1}\ast\mathcal{X}, \mathcal{Y} \ast \mathcal{F}_{2})$.
In particular, $\Phi$ and $\Psi$ preserve hearts, that is, for $\stors [t, t'] \subseteq \stors[t_{1}, t_{2}]$ and $\stors [x, x'] \subseteq \stors\mathcal{H}_{[t_{1}, t_{2}]}$, we have $\mathcal{H}_{[t, t']}=\mathcal{H}_{[\Phi(t),\Phi(t')]}$ and $\mathcal{H}_{[x,x']}=\mathcal{H}_{[\Psi(x), \Psi(x')]}$.
\end{theorem}

Figure \ref{fig} illustrates the relation between various subcategories of $\mathcal{C}$ appearing in the statement of Theorem \ref{mainthm}.

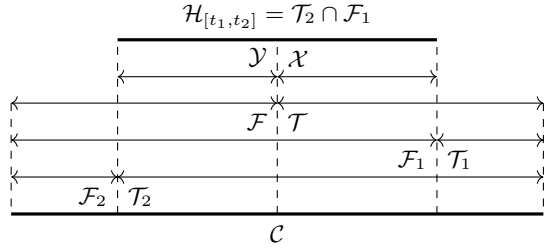
\begin{figure}[htp]
\begin{tikzpicture}[scale = 0.7, every node/.style={font=\footnotesize}]
\draw[very thick] (-5,0)-- (5,0) node[midway, below] {$\mathcal{C}$};
\draw[dashed] (-5,2.1) -- (-5,0);
\draw[dashed] (5,2.1) -- (5,0);
\draw[very thick] (-3,3.3)-- (3,3.3)
node[midway, above] {$\mathcal{H}_{[t_1,t_2]} = \mathcal{T}_2 \cap \mathcal{F}_1$};
\draw[dashed] (-3,0) -- (-3,3.3);
\draw[<->] (5,0.7) -- (-3,0.7) node[below right] {$\mathcal{T}_2$};
\draw[<->] (-5,0.7) -- (-3,0.7) node[below left] {$\mathcal{F}_2$};
\draw[dashed] (3,0) -- (3,3.3);
\draw[<->] (5,1.4) -- (3,1.4) node[below right] {$\mathcal{T}_1$};
\draw[<->] (-5,1.4) -- (3,1.4) node[below left] {$\mathcal{F}_1$};
\draw[dashed] (0,0) -- (0,3.3);
\draw[<->] (5,2.1) -- (0,2.1) node[below right] {$\mathcal{T}$};
\draw[<->] (-5,2.1) -- (0,2.1) node[below left] {$\mathcal{F}$};
\draw[<->] (3,2.6) -- (0,2.6) node[above right] {$\mathcal{X}$};
\draw[<->] (-3,2.6) -- (0,2.6) node[above left] {$\mathcal{Y}$};
\end{tikzpicture}
\caption{Illustration of Theorem \ref{mainthm}}\label{fig}
\end{figure}

Before proving Theorem \ref{mainthm}, we give a remark on brick labeling.

\begin{remark}
Let $\mathcal{A}$ be a length abelian category and let $t_{1}\le t_{2}\in\tors\mathcal{A}$.
By a similar argument of \cite[Theorem 3.3(b)]{DIRRT}, we can show that there exists an arrow $a:t\to t'$ in the Hasse quiver of $\stors\mathcal{H}_{[t_{1},t_{2}]}$ if and only if there exists a unique brick $S_{a}$ in $\mathcal{H}_{[t_{1},t_{2}]}$ (up to isomorphism).
Hence we can introduce brick labeling in the poset of $s$-torsion pairs in the heart $\mathcal{H}_{[t_{1},t_{2}]}$.
\end{remark}

In the following, we give a proof of Theorem \ref{mainthm}.
Let $\mathcal{C}$ be an extriangulated category with a negative first extension and fix two $s$-torsion pairs $t_{1}:=(\mathcal{T}_{1},\mathcal{F}_{1})\le t_{2}:=(\mathcal{T}_{2},\mathcal{F}_{2})$ in $\mathcal{C}$.
We start with giving the following lemma, which is frequently used in the rest of this paper.

\begin{lemma}\label{lem_heart}
Let $(\mathcal{T}',\mathcal{F}')\leq (\mathcal{T},\mathcal{F})\in\stors \mathcal{C}$.
Then the following statements hold.
\begin{itemize}
\item[(1)] $\mathcal{T} = \mathcal{T}' \ast (\mathcal{T} \cap \mathcal{F}')$.
\item[(2)] $\mathcal{F}' = (\mathcal{T}\cap \mathcal{F}') \ast \mathcal{F}$.
\end{itemize}
\end{lemma}

\begin{proof}
We only prove (1); the proof of (2) is similar.
Since $\mathcal{T}$ is extension-closed, we have $\mathcal{T} \supseteq \mathcal{T}' \ast (\mathcal{T} \cap \mathcal{F}')$.
We prove the converse inclusion.
Let $C \in \mathcal{T}$.
By $\mathcal{C} = \mathcal{T}' \ast \mathcal{F}'$, there exists an $\mathfrak{s}$-conflation $T'\rightarrow C\rightarrow F'\dashrightarrow$ such that $T' \in \mathcal{T}'$ and $F' \in \mathcal{F}'$.
It is enough to show that $F' \in \mathcal{T}$.
For each $F \in \mathcal{F}$, we obtain an exact sequence
\[
\begin{tikzcd}
\mathbb{E}^{-1}(T', F) \rar & \mathcal{C}(F', F) \rar & \mathcal{C}(C,F).
\end{tikzcd}
\]
Since $(\mathcal{T}, \mathcal{F})$ is an $s$-torsion pair, the left-side hand and right-hand side vanish.
Hence we have $F'\in {}^{\perp}\mathcal{F}$.
The assertion follows from Proposition \ref{prop_orth}(2).
\end{proof}

We show that $\Phi$ and $\Psi$ are well-defined.

\begin{proposition}\label{prop_phi}
If $(\mathcal{T}, \mathcal{F}) \in \stors [t_{1}, t_{2}]$, then $(\mathcal{T} \cap \mathcal{F}_{1}, \mathcal{T}_{2} \cap \mathcal{F})$ is an $s$-torsion pair in $\mathcal{H}_{[t_{1}, t_{2}]}$.
\end{proposition}

\begin{proof}
For simplicity, we put $\mathcal{H}:=\mathcal{H}_{[t_{1}, t_{2}]}$.
Let $(\mathcal{T}, \mathcal{F}) \in \stors [t_{1}, t_{2}]$.
Then we clearly have $\mathcal{T} \cap \mathcal{F}_{1}, \mathcal{T}_{2} \cap \mathcal{F} \subseteq \mathcal{H}$.
We show (STP1), that is, $\mathcal{H}=(\mathcal{T} \cap \mathcal{F}_{1}) \ast (\mathcal{T}_{2} \cap \mathcal{F})$.
Since $\mathcal{H}$ is extension-closed, we obtain that  $\mathcal{H}\supseteq (\mathcal{T} \cap \mathcal{F}_{1}) \ast ( \mathcal{T}_{2} \cap \mathcal{F})$.
We show the converse inclusion.
By Lemma \ref{lem_heart}, $\mathcal{H}=\mathcal{T}_{2}\cap \mathcal{F}_{1}= (\mathcal{T}\ast (\mathcal{T}_{2}\cap \mathcal{F}))\cap ((\mathcal{T}\cap\mathcal{F}_{1})\ast \mathcal{F})$.
It follows from Proposition \ref{prop_canseq} that $\mathfrak{s}$-conflations in (STP1) uniquely exist.
Since $(\mathcal{T},\mathcal{F})$ is an $s$-torsion pair, we obtain that for each $H\in \mathcal{H}$, there uniquely exists an $\mathfrak{s}$-conflation $T\rightarrow H\rightarrow F\dashrightarrow $ such that $T\in\mathcal{T}\cap \mathcal{F}_{1}$ and $F\in\mathcal{T}_{2}\cap \mathcal{F}$. Hence $H\in (\mathcal{T}\cap \mathcal{F}_{1})\ast( \mathcal{T}_{2}\cap \mathcal{F}$).
The conditions (STP2) and (STP3) follow from the fact that $(\mathcal{T}, \mathcal{F})$ is an $s$-torsion pair in $\mathcal{C}$.
\end{proof}

\begin{proposition}\label{prop_psi}
If $(\mathcal{X}, \mathcal{Y})$ is an $s$-torsion pair in $\mathcal{H}_{[t_{1}, t_{2}]}$, then $(\mathcal{T}_{1} \ast \mathcal{X}, \mathcal{Y} \ast \mathcal{F}_{2}) \in \stors [t_{1}, t_{2}]$.
\end{proposition}

\begin{proof}
Let $(\mathcal{X}, \mathcal{Y}) \in \stors \mathcal{H}_{[t_{1}, t_{2}]}$.
We show (STP1), that is, $\mathcal{C}=(\mathcal{T}_{1} \ast \mathcal{X}) \ast (\mathcal{Y} \ast \mathcal{F}_{2})$.
By Lemma \ref{lem_heart}(1), we have
\begin{align}
(\mathcal{T}_{1}\ast \mathcal{X}) \ast(\mathcal{Y} \ast \mathcal{F}_{2})=(\mathcal{T}_{1}\ast \mathcal{H}_{[t_{1}, t_{2}]}) \ast \mathcal{F}_{2}=(\mathcal{T}_{1}\ast (\mathcal{T}_{2}\cap\mathcal{F}_{1})) \ast \mathcal{F}_{2}=\mathcal{T}_{2} \ast \mathcal{F}_{2} =\mathcal{C}.\notag
\end{align}
We show (STP2) and (STP3).
Clearly $\mathcal{C}(\mathcal{X}, \mathcal{Y})=0$ and $\mathbb{E}^{-1}(\mathcal{X}, \mathcal{Y})=0$ hold.
Since $(\mathcal{T}_{2}, \mathcal{F}_{2})$ is an $s$-torsion pair, we obtain that $\mathcal{C}(\mathcal{X}, \mathcal{F}_{2})=0$ and $\mathbb{E}^{-1}(\mathcal{X}, \mathcal{F}_{2})=0$.
This implies that $\mathcal{C}(\mathcal{X}, \mathcal{Y} \ast \mathcal{F}_{2})=0$ by Proposition \ref{prop_longex} and $\mathbb{E}^{-1}(\mathcal{X},\mathcal{Y}\ast \mathcal{F}_{2})=0$ by (NE2).
Since $\mathcal{Y} \ast \mathcal{F}_{2} \subseteq \mathcal{F}_{1}$ and $(\mathcal{T}_{1}, \mathcal{F}_{1})$ is an $s$-torsion pair, we have that $\mathcal{C}(\mathcal{T}_{1}, \mathcal{Y}\ast \mathcal{F}_{2})=0$ and $\mathbb{E}^{-1}(\mathcal{T}_{1}, \mathcal{Y}\ast \mathcal{F}_{2})=0$.
Therefore we obtain that $\mathcal{C}(\mathcal{T}_{1} \ast \mathcal{X}, \mathcal{Y} \ast \mathcal{F}_{2})=0$ by Proposition \ref{prop_longex} and $\mathbb{E}^{-1}(\mathcal{T}_{1} \ast \mathcal{X}, \mathcal{Y} \ast \mathcal{F}_{2})=0$ by (NE2).
Thus $(\mathcal{T}_{1} \ast \mathcal{X}, \mathcal{Y} \ast \mathcal{F}_{2}) \in \stors\mathcal{C}$ holds.
Since $\mathcal{T}_{1} \subseteq \mathcal{T}_{1} \ast \mathcal{X} \subseteq \mathcal{T}_{2}$ clearly holds, we have the assertion.
\end{proof}

Now we are ready to prove Theorem \ref{mainthm}.

\begin{proof}[Proof of Theorem \ref{mainthm}]
By Propositions \ref{prop_phi} and \ref{prop_psi}, the maps $\Phi$ and $\Psi$ are well-defined.
Moreover, it is clear that these maps are order-preserving.
We show that $\Phi$ and $\Psi$ are mutually inverse isomorphisms.
Let $(\mathcal{T}, \mathcal{F}) \in \stors[t_{1}, t_{2}]$.
Then
\begin{align}
\Psi  \Phi ((\mathcal{T}, \mathcal{F}))=(\mathcal{T}_{1} \ast (\mathcal{T}\cap\mathcal{F}_{1}),(\mathcal{T}_{2} \cap \mathcal{F})\ast \mathcal{F}_{2})=(\mathcal{T}, \mathcal{F}),   \notag
\end{align}
where the last equality follows from Lemma \ref{lem_heart}.
Let $(\mathcal{X}, \mathcal{Y})$ be an $s$-torsion pair in $\mathcal{H}_{[t_{1}, t_{2}]}$.
Then we have
\begin{align}
\Phi \Psi ((\mathcal{X}, \mathcal{Y}))=((\mathcal{T}_{1} \ast \mathcal{X}) \cap \mathcal{F}_{1}, \mathcal{T}_{2} \cap(\mathcal{Y} \ast \mathcal{F}_{2})). \notag
\end{align}
Since a torsion-free class is uniquely determined by a torsion class, it is enough to show that $(\mathcal{T}_{1} \ast \mathcal{X}) \cap \mathcal{F}_{1}=\mathcal{X}$.
Clearly $(\mathcal{T}_{1} \ast \mathcal{X}) \cap \mathcal{F}_{1} \supseteq \mathcal{X}$ holds.
We prove the converse inclusion.
By $\mathcal{C}(\mathcal{T}_{1}, \mathcal{F}_{1})=0$ and $\mathcal{C}(\mathcal{X}, \mathcal{Y})=0$, we obtain
\begin{align}
(\mathcal{T}_{1} \ast \mathcal{X})\cap \mathcal{F}_{1}\subseteq \mathcal{H}_{[t_{1}, t_{2}]} \cap {}^{\perp} \mathcal{Y}=\mathcal{X}, \notag
\end{align}
where the last equality follows from Proposition \ref{prop_orth}(2).

We show that $\Phi$ and $\Psi$ preserve the hearts.
Since $\Phi$ and $\Psi$ are mutually inverse isomorphisms, it is enough to show that $\mathcal{H}_{[t, t']}=\mathcal{H}_{[\Phi(t), \Phi(t')]}$ holds for each  $\stors [t, t'] \subseteq \stors [t_{1}, t_{2}]$, where $t:=(\mathcal{T}, \mathcal{F})$ and $t':=(\mathcal{T}', \mathcal{F}')$.
By definition, we have
\begin{align}
\mathcal{H}_{[\Phi(t), \Phi(t')]}=(\mathcal{T}'\cap \mathcal{F}_{1}) \cap (\mathcal{T}_{2} \cap \mathcal{F})=\mathcal{T}'\cap\mathcal{F}=\mathcal{H}_{[t,t']}. \notag
\end{align}
This finishes the proof.
\end{proof}

We give two applications of Theorem \ref{mainthm}.
Let $\mathcal{D}$ be a triangulated category.
For $t$-structures $(\mathcal{U}_{1}, \mathcal{V}_{1}), (\mathcal{U}_{2}, \mathcal{V}_{2})$ on $\mathcal{D}$ with $\mathcal{U}_{1} \subseteq \mathcal{U}_{2}$, let
\begin{align}
\tstr [(\mathcal{U}_{1}, \mathcal{V}_{1}), (\mathcal{U}_{2}, \mathcal{V}_{2})]:= \{\textnormal{$(\mathcal{U}, \mathcal{V})$: a $t$-structure on $\mathcal{D}$} \mid \mathcal{U}_{1} \subseteq \mathcal{U} \subseteq \mathcal{U}_{2}\}. \notag
\end{align}
By Theorem \ref{mainthm}, we have the following result, which recovers Theorem \ref{intro_thm1}.

\begin{corollary}\label{recover-tstr}
Let $\mathcal{D}$ be a triangulated category.
For $i=1,2$, let $(\mathcal{U}_{i},\mathcal{V}_{i})$ be a $t$-structure on $\mathcal{D}$ with $\mathcal{U}_{1} \subseteq \mathcal{U}_{2}$ and $\mathcal{H}:=\mathcal{U}_{2} \cap \mathcal{V}_{1}$.
Then there exist mutually inverse isomorphisms of posets
\[
\begin{tikzcd}
\tstr {[}(\mathcal{U}_{1},\mathcal{V}_{1}),(\mathcal{U}_{2},\mathcal{V}_{2}){]} \rar[shift left, "\Phi"] & \stors\mathcal{H}, \lar[shift left, "\Psi"]
\end{tikzcd}
\]
where $\Phi(\mathcal{T},\mathcal{F}):=(\mathcal{T}\cap\mathcal{V}_{1}, \mathcal{U}_{2} \cap\mathcal{F})$ and $\Psi(\mathcal{X},\mathcal{Y}):=(\mathcal{U}_{1}\ast\mathcal{X}, \mathcal{Y} \ast \mathcal{V}_{2})$.
In addition, if $\Sigma \mathcal{U}_2 \subseteq \mathcal{U}_1$ holds, then $\mathcal{H}$ becomes an exact category by the induced extriangulated structure, and we have $\stors \mathcal{H} = \tors\mathcal{H}$.
\end{corollary}

\begin{proof}
We regard $\mathcal{D}$ as the extriangulated category with the negative first extension $\mathbb{E}^{-1}(-, -):=\mathcal{D}(-, \Sigma^{-1}-)$.
By Example \ref{ex_tstr}, we have
\begin{align}
\tstr [(\mathcal{U}_{1},\mathcal{V}_{1}),(\mathcal{U}_{2},\mathcal{V}_{2})]=\stors [(\mathcal{U}_{1},\mathcal{V}_{1}),(\mathcal{U}_{2},\mathcal{V}_{2})].\notag
\end{align}
Hence the former assertion follows from Theorem \ref{mainthm}.

We prove the latter assertion.
Assume that $\Sigma\mathcal{U}_2 \subseteq \mathcal{U}_1$ holds.
We claim that the restriction of $\mathbb{E}^{-1}$ on $\mathcal{H}$ vanishes.
Indeed, for $C, A \in \mathcal{H} = \mathcal{U}_2 \cap \mathcal{V}_1$, we have $\mathbb{E}^{-1}(C,A) = \mathcal{D}(C,\Sigma^{-1} A) \cong \mathcal{D}(\Sigma C, A) = 0$ by $\Sigma C \in \Sigma\mathcal{U}_2 \subseteq \mathcal{U}_1$ and $A \in \mathcal{V}_1$.
Therefore $\mathcal{H}$ becomes an exact category by Proposition \ref{prop_exactcat}.
Since the negative first extension on $\mathcal{H}$ vanishes, Example \ref{ex_extor} shows $\stors \mathcal{H} = \tors \mathcal{H}$.
\end{proof}

Next, we apply Theorem \ref{mainthm} to exact categories.
Let $\mathcal{E}$ be an exact category.
For two torsion pairs $(\mathcal{T}_{1}, \mathcal{F}_{1})$ and $ (\mathcal{T}_{2}, \mathcal{F}_{2})$ in $\mathcal{E}$ with $\mathcal{T}_{1} \subseteq \mathcal{T}_{2}$, let
\begin{align}
\tors [(\mathcal{T}_{1}, \mathcal{F}_{1}), (\mathcal{T}_{2}, \mathcal{F}_{2})]:= \{\textnormal{$(\mathcal{T}, \mathcal{F})$}\in \tors \mathcal{E}\mid \mathcal{T}_{1} \subseteq \mathcal{T} \subseteq \mathcal{T}_{2}\}. \notag
\end{align}
The following corollary is a further generalization of results in \cite[Theorem 3.12]{J15}, \cite[Theorem 4.2]{AP} and \cite[Theorem A]{T}, where the abelian category cases are proved.

\begin{corollary}\label{recover-tors}
Let $\mathcal{E}$ be an exact category.
For $i=1, 2$, let $(\mathcal{T}_{i}, \mathcal{F}_{i})$ be a torsion pair in $\mathcal{E}$ with $\mathcal{T}_{1} \subseteq \mathcal{T}_{2}$ and $\mathcal{H}:=\mathcal{T}_{2} \cap \mathcal{F}_{1}$.
Then there exist mutually inverse isomorphisms of posets
\[
\begin{tikzcd}
\tors{[}(\mathcal{T}_{1},\mathcal{F}_{1}),(\mathcal{T}_{2},\mathcal{F}_{2}){]} \rar[shift left, "\Phi"] & \tors \mathcal{H}, \lar[shift left, "\Psi"]
\end{tikzcd}
\]
where $\Phi(\mathcal{T},\mathcal{F}):=(\mathcal{T}\cap\mathcal{F}_{1},\mathcal{T}_{2} \cap \mathcal{F})$ and $\Psi(\mathcal{X},\mathcal{Y}):=(\mathcal{T}_{1}\ast\mathcal{X}, \mathcal{Y}\ast \mathcal{F}_{2})$.
\end{corollary}

\begin{proof}
Since extension-closed subcategories of exact categories are exact categories, $\mathcal{H}$ is an exact category.
By Example \ref{ex_neg}(2), we can regard an exact category as the extriangulated category such that its negative first extension vanishes.
Thus we have
\begin{align}
\tors [(\mathcal{T}_{1},\mathcal{F}_{1}),(\mathcal{T}_{2},\mathcal{F}_{2})]&=\stors [(\mathcal{T}_{1},\mathcal{F}_{1}),(\mathcal{T}_{2},\mathcal{F}_{2})], \notag \\
\tors \mathcal{H}&=\stors \mathcal{H} \notag
\end{align}
by Example \ref{ex_extor}.
Hence the assertion follows from Theorem \ref{mainthm}.
\end{proof}

In the following, we study how functorially finite $s$-torsion pairs behave under the isomorphisms of Theorem \ref{mainthm}.
First we recall the notion of functorially finite subcategories.
A subcategory $\mathcal{C}'$ of $\mathcal{C}$ is called a \emph{contravariantly finite} subcategory if for each $X \in \mathcal{C}$, there exists a right $\mathcal{C}'$-approximation of $X$.
Dually, \emph{covariantly finite} subcategories are defined.
Moreover, we say that $\mathcal{C}'$ is \emph{functorially finite} in $\mathcal{C}$ if it is contravariantly finite and covariantly finite in $\mathcal{C}$.
If $(\mathcal{T}, \mathcal{F})$ is an $s$-torsion pair, then $\mathcal{T}$ is a contravariantly finite subcategory of $\mathcal{C}$ and $\mathcal{F}$ is a covariantly finite subcategory of $\mathcal{C}$ by (STP1) and (STP2).
Then it is natural to consider the following condition on $s$-torsion pairs.

\begin{definition}
An $s$-torsion pair $(\mathcal{T},\mathcal{F})$ is said to be \emph{functorially finite} in $\mathcal{C}$ if both $\mathcal{T}$ and $\mathcal{F}$ are functorially finite in $\mathcal{C}$.
\end{definition}

Let $\fstors \mathcal{C}$ denote the set of functorially finite $s$-torsion pairs in $\mathcal{C}$.
For two functorially finite $s$-torsion pairs $t_{1}:=(\mathcal{T}_{1},\mathcal{F}_{1})\leq t_{2}:=(\mathcal{T}_{2},\mathcal{F}_{2})$, let $\fstors[t_{1},t_{2}]:= \stors[t_{1},t_{2}]\cap \fstors\mathcal{C}$.
The following proposition shows that the map $\Phi$ in Theorem \ref{mainthm} preserves functorially finiteness.

\begin{proposition}\label{prop_ff}
For $i=1,2$, let $t_{i}:=(\mathcal{T}_{i},\mathcal{F}_{i})\in \stors\mathcal{C}$ with $t_{1}\le t_{2}$.
Assume that $\mathcal{F}_{1}$ and $\mathcal{T}_{2}$ are functorially finite in $\mathcal{C}$.
Then the following statements hold.
\begin{itemize}
\item[(1)] If $(\mathcal{T}, \mathcal{F}) \in \fstors [t_{1}, t_{2}]$, then $(\mathcal{T} \cap \mathcal{F}_{1}, \mathcal{T}_{2} \cap \mathcal{F})$ is functorially finite in $\mathcal{C}$.
In particular, $\mathcal{H}_{[t_{1}, t_{2}]}$ is functorially finite in $\mathcal{C}$.
\item[(2)] If $(\mathcal{X}, \mathcal{Y}) \in \fstors \mathcal{H}_{[t_{1}, t_{2}]}$, then $\mathcal{X}$ and $\mathcal{Y}$ are functorially finite in $\mathcal{C}$.
In particular,
\begin{align}
\fstors \mathcal{H}_{[t_{1}, t_{2}]}=\{t\in \stors \mathcal{H}_{[t_{1}, t_{2}]} \mid t \mathrm{\; is\; functorially\; finite\; in\; }\mathcal{C}\}.\notag
\end{align}
\end{itemize}
\end{proposition}

\begin{proof}
(1) Let $(\mathcal{T}, \mathcal{F}) \in \fstors [t_{1}, t_{2}]$.
We prove only that $\mathcal{T}\cap \mathcal{F}_{1}$ is functorially finite in $\mathcal{C}$ since the proof for $\mathcal{T}_{2} \cap \mathcal{F}$ is similar.

We show that $\mathcal{T} \cap \mathcal{F}_{1}$ is covariantly finite in $\mathcal{T}$.
Let $M \in \mathcal{T}$.
By Lemma \ref{lem_heart}(1), there exists an $\mathfrak{s}$-conflation $T_{1}\rightarrow M\xrightarrow{g}F_{1}\dashrightarrow$ such that $T_{1} \in \mathcal{T}_{1}$ and $F_{1} \in \mathcal{T} \cap \mathcal{F}_{1}$.
Then we have an exact sequence
\[
\begin{tikzcd}
\mathcal{C}(F_{1}, W) \rar["{\mathcal{C}(g, W)}"] & \mathcal{C}(M, W) \rar & \mathcal{C}(T_{1},W)=0
\end{tikzcd}
\]
for each $W \in \mathcal{T}\cap \mathcal{F}_{1}$.
Thus $g$ is a left $(\mathcal{T} \cap \mathcal{F}_{1})$-approximation of $M$, and hence $\mathcal{T}\cap \mathcal{F}_{1}$ is covariantly finite in $\mathcal{T}$.
Similarly, we can show that $\mathcal{T}\cap \mathcal{F}_{1}$ is contravariantly finite in $\mathcal{F}_{1}$.
Since $\mathcal{T}$ is covariantly finite in $\mathcal{C}$ and $\mathcal{F}_{1}$ is contravariantly finite in $\mathcal{C}$, we obtain that $\mathcal{T}\cap \mathcal{F}_{1}$ is functorially finite in $\mathcal{C}$.

(2) Let $(\mathcal{X}, \mathcal{Y}) \in \fstors \mathcal{H}_{[t_{1}, t_{2}]}$.
Then $\mathcal{X}$ and $\mathcal{Y}$ are functorially finite in $\mathcal{H}_{[t_{1}, t_{2}]}$.
By (1), $\mathcal{H}_{[t_{1}, t_{2}]}$ is functorially finite in $\mathcal{C}$.
Hence we have the assertion.
\end{proof}

We state that the isomorphisms in Theorem \ref{mainthm} give isomorphisms between functorially finite $s$-torsion pairs.
For an extriangulated category $\mathcal{C}$, we call $P \in \mathcal{C}$ a \emph{projective} object if $\mathbb{E}(P, \mathcal{C})=0$.
We say that $\mathcal{C}$ has \emph{enough projectives} if for each $C \in \mathcal{C}$, there exists an $\mathfrak{s}$-conflation $A \to P \to C \dashrightarrow$ such that $P$ is a projective object in $\mathcal{C}$.
Dually, we define \emph{injective} objects and that $\mathcal{C}$ has \emph{enough injectives}.

\begin{theorem}
Let $\mathcal{C}$ be an extriangulated category with a negative first extension.
For $i=1,2$, let $t_{i}:=(\mathcal{T}_{i},\mathcal{F}_{i})\in \fstors\mathcal{C}$ with $t_{1}\le t_{2}$.
Assume that $\mathcal{C}$ has enough projectives and enough injectives.
Then there exist mutually inverse isomorphisms of posets
\[
\begin{tikzcd}
\fstors {[t_{1}, t_{2}]} \rar[shift left, "\Phi"] & \fstors\mathcal{H}_{[t_{1}, t_{2}]}, \lar[shift left, "\Psi"]
\end{tikzcd}
\]
where $\Phi(\mathcal{T},\mathcal{F}):=(\mathcal{T}\cap\mathcal{F}_{1}, \mathcal{T}_{2} \cap \mathcal{F})$ and $\Psi(\mathcal{X},\mathcal{Y}):=(\mathcal{T}_{1}\ast\mathcal{X}, \mathcal{Y} \ast \mathcal{F}_{2})$.
\end{theorem}

\begin{proof}
By Theorem \ref{mainthm}, it is enough to show that $\Phi$ and $\Psi$ are well-defined.
It follows from Proposition \ref{prop_ff}(1) that $\Phi$ is well-defined.
Let $(\mathcal{X}, \mathcal{Y}) \in \fstors \mathcal{H}_{[t_{1}, t_{2}]}$.
By Proposition \ref{prop_ff}(2), $\mathcal{X}$ is covariantly finite in $\mathcal{C}$.
Since $\mathcal{C}$ has enough projectives, it follows from \cite[Theorem 3.3]{H19} that $\mathcal{T}_{1}\ast\mathcal{X}$ is covariantly finite in $\mathcal{C}$, and hence it is functorially finite in $\mathcal{C}$.
Similarly, we can show that $\mathcal{Y} \ast \mathcal{F}_{2}$ is functorially finite in $\mathcal{C}$.
Hence $\Psi$ is well-defined. This finishes the proof.
\end{proof}

\subsection{Example: successor-closed subsets of quivers}
Let $\Lambda$ be a finite-dimensional algebra over a field $k$, and let $\operatorname{\mathsf{mod}}\Lambda$ denote the category of finitely generated right $\Lambda$-modules.
Since $\operatorname{\mathsf{mod}}\Lambda$ is abelian, it has a natural exact structure. Therefore it can be regarded as the extriangulated category with the trivial negative first extension by Example \ref{ex_neg}(2).

In this subsection, we introduce different negative first extensions in $\operatorname{\mathsf{mod}} \Lambda$ when the global dimension $\mathop{\mathrm{gl.dim}} \Lambda$ is finite.
Moreover, if $\Lambda$ is hereditary, then we give a combinatorial interpretation of Theorem \ref{mainthm} for this new negative first extension.

We start with giving the following example which shows that $\operatorname{\mathsf{mod}} \Lambda$ admits a non-trivial negative first extension.

\begin{example}\label{ex_fingldim}
Let $\Lambda$ be a finite-dimensional $k$-algebra with $\mathop{\mathrm{gl.dim}} \Lambda \leq n$. Define a bifunctor $\mathbb{E}^{-1}(-,-) : (\operatorname{\mathsf{mod}}\Lambda)^{\mathrm{op}} \times \operatorname{\mathsf{mod}}\Lambda \to \mathcal{A}b $ as
\begin{align}
\mathbb{E}^{-1}(-,-) := \operatorname{Ext}_\Lambda^n(-,-). \notag
\end{align}
Then $\mathbb{E}^{-1}(C,-)$ and $\mathbb{E}^{-1}(-,A)$ are right exact functors by $\mathop{\mathrm{gl.dim}} \Lambda \leq n$.
Thus $\mathbb{E}^{-1}=\operatorname{Ext}_{\Lambda}^{n}$ gives a negative first extension structure on $\operatorname{\mathsf{mod}}\Lambda$ together with $\delta_\sharp^{-1} = 0$ and $\delta_{-1}^\sharp = 0$ for all $\delta$.
\end{example}

This negative first extension structure naturally appears by considering periodic derived categories.
We refer the reader to \cite{G, Z14, Sa} for the details on the periodic derived category.

\begin{remark}\label{rem_periodic}
Fix an integer $n\ge 1$.
Let $\Lambda$ be a finite-dimensional algebra with $\mathop{\mathrm{gl.dim}} \Lambda \leq n$. Consider the $(n+1)$-st periodic derived category $D_{n+1}(\Lambda)$.
We can check that $\operatorname{\mathsf{mod}} \Lambda$ is equivalent to an extension-closed subcategory of $D_{n+1}(\Lambda)$.
Thus it has the induced extriangulated structure and the negative first extension structure by Example \ref{ex_neg}(3).
For all $X,Y\in \operatorname{\mathsf{mod}}\Lambda$, we have isomorphisms
\begin{align}\label{isom320}
\mathbb{E}^{-1}(X,Y)=D_{n+1}(\Lambda)(X,\Sigma^{-1}Y)\cong  D_{n+1}(\Lambda)(X,\Sigma^{n}Y)\cong\operatorname{Ext}_{\Lambda}^{n}(X,Y)
\end{align}
as abelian groups, where the middle isomorphism follows from $\Sigma^{n+1}Y\cong Y$.
Moreover, if $n$ is odd, then $D_{n+1}(\Lambda)$ has a natural isomorphism $\Sigma^{n+1}\cong \mathsf{id}_{D_{n+1}(\Lambda)}$.
Thus \eqref{isom320} is also a natural isomorphism in both variables.
Hence this negative first extension structure is the same as Example \ref{ex_fingldim}.
\end{remark}

We recall the negative first extension structure in Example \ref{ex_nakayama}.
Let $\Lambda$ and $\mathcal{A}$ be in Example \ref{ex_nakayama}.
Then $\operatorname{\underline{\mathsf{mod}}}\Lambda$ is triangle equivalent to the $2$-periodic derived category $D_{2}(kQ)$, where $Q$ is a quiver of type $A_{2}$ (see \cite[Example 5.6]{Sa} for example).
Hence, the negative first extension structure of $\mathcal{A}$ comes from that of  $D_{2}(kQ)$.

Let us focus on the case $n=1$.
Namely, we assume that $\Lambda$ is hereditary.
Thus Example \ref{ex_fingldim} shows that $\operatorname{\mathsf{mod}} \Lambda$ has a non-trivial negative first extension structure defined by $\mathbb{E}^{-1} = \operatorname{Ext}_{\Lambda}^1$.
Let $\mathcal{C}$ denote the extriangulated category $\operatorname{\mathsf{mod}} \Lambda$ with this negative first extension except for regarding $\operatorname{\mathsf{mod}} \Lambda$ as an abelian category.
First, we give a characterization of $s$-torsion pairs in $\mathcal{C}$.
Note that every $s$-torsion pair ($\mathcal{T},\mathcal{F}$) in $\mathcal{C}$ is a torsion pair in $\operatorname{\mathsf{mod}} \Lambda$ (in the usual sense) by (STP1) and (STP2).
As for (STP3), we make the following general observation.

\begin{proposition}\label{prop_quiverserre}
Let $\Lambda$ be an arbitrary finite dimensional algebra and $(\mathcal{T}, \mathcal{F})$ a torsion pair in $\operatorname{\mathsf{mod}} \Lambda$ (in the usual sense).
Then the following statements are equivalent.
\begin{itemize}
\item[(1)] $\operatorname{Ext}_{\Lambda}^{1}(\mathcal{T}, \mathcal{F})=0$.
\item[(2)] $\mathcal{T}$ and $\mathcal{F}$ are Serre subcategories, that is, $\mathcal{T}$ is closed under submodules and $\mathcal{F}$ is closed under quotients.
\end{itemize}
\end{proposition}

\begin{proof}
(1)$\Rightarrow$(2):
We only show that $\mathcal{T}$ is closed under submodules; the proof for $\mathcal{F}$ is similar.
Assume that we have an exact sequence
\[
\begin{tikzcd}
0 \rar & A \rar & T \rar & C \rar & 0
\end{tikzcd}
\]
with $T \in \mathcal{T}$.
Then $C \in \mathcal{T}$ holds since $\mathcal{T}$ is closed under quotients.
By applying $\operatorname{Hom}_{\Lambda}(-, F)$ with $F \in \mathcal{F}$, we obtain an exact sequence
\[
\begin{tikzcd}
\operatorname{Hom}_{\Lambda}(T,F) \rar & \operatorname{Hom}_{\Lambda}(A,F) \rar & \operatorname{Ext}_{\Lambda}^{1}(C,F).
\end{tikzcd}
\]
Then $\operatorname{Hom}_{\Lambda}(T,F) = 0$ by $\operatorname{Hom}_{\Lambda}(\mathcal{T},\mathcal{F}) = 0$, and $\operatorname{Ext}_{\Lambda}^{1}(C,F) = 0$ by $\operatorname{Ext}_{\Lambda}^{1}(\mathcal{T},\mathcal{F}) = 0$ and $C \in \mathcal{T}$. Thus we obtain $\operatorname{Hom}_{\Lambda}(A,\mathcal{F}) = 0$, and hence $A \in {}^\perp \mathcal{F} = \mathcal{T}$.

(2)$\Rightarrow$(1):
We claim that $\mathcal{T}$ is closed under projective covers.
Let $T \in \mathcal{T}$ and consider an exact sequence
\[
\begin{tikzcd}
0 \rar & \Omega T \rar & P \rar["\pi"] & T \rar & 0,
\end{tikzcd}
\]
where $\pi$ is a projective cover of $T$.
We show $P \in {}^\perp \mathcal{F}$.
Since $\mathcal{F}$ is a Serre subcategory, it suffices to show that $\operatorname{Hom}_{\Lambda}(P,S) = 0$ for every simple $\Lambda$-module $S \in \mathcal{F}$. Now since $\Omega T \subseteq \operatorname{rad} P$, we have an isomorphism $\operatorname{Hom}_{\Lambda}(T,S) \cong \operatorname{Hom}_{\Lambda}(P,S)$.
Hence $\operatorname{Hom}_{\Lambda}(P,S) = 0$ follows from $\operatorname{Hom}_{\Lambda}(T,S) = 0$. Thus we obtain $P \in {}^\perp \mathcal{F}=\mathcal{T}$.
Since $\mathcal{T}$ is a Serre subcategory, $\Omega T$ also belongs to $\mathcal{T}$. Take any $F \in \mathcal{F}$. Then we have an exact sequence
\[
\begin{tikzcd}
\operatorname{Hom}_{\Lambda}(\Omega T,F) \rar & \operatorname{Ext}_{\Lambda}^{1}(T,F) \rar & \operatorname{Ext}_{\Lambda}^{1}(P,F)=0.
\end{tikzcd}
\]
Thus $\operatorname{Ext}_{\Lambda}^{1}(T,F) = 0$ holds by $\operatorname{Hom}_{\Lambda}(\Omega T,F) = 0$.
\end{proof}

The following description of $s$-torsion pairs in $\mathcal{C}$ immediately follows from Proposition \ref{prop_quiverserre}.

\begin{corollary}\label{cor_stros_serre}
Let $\Lambda$ be a hereditary algebra and let $(\mathcal{T},\mathcal{F})$ be a pair of subcategories of $\mathcal{C} = \operatorname{\mathsf{mod}} \Lambda$.
Then the following statements are equivalent.
\begin{itemize}
\item[(1)] $(\mathcal{T},\mathcal{F})$ is an $s$-torsion pair in $\mathcal{C}$.
\item[(2)] $(\mathcal{T},\mathcal{F})$ is a torsion pair in $\operatorname{\mathsf{mod}} \Lambda$ (in the usual sense) such that $\mathcal{T}$ and $\mathcal{F}$ are Serre subcategories of $\operatorname{\mathsf{mod}} \Lambda$.
\end{itemize}
\end{corollary}

In the rest of this subsection, we assume that $\Lambda$ is the path algebra of a finite acyclic quiver $Q$.
Let $Q_{0}$ denote the set of all vertices in $Q$.
For a vertex $i \in Q_0$, let $e_i$ denote the primitive idempotent of $\Lambda$ corresponding to $i$. We define $P(i):= e_i \Lambda$ and $S(i):= e_i \Lambda / \operatorname{\mathrm{rad}} (e_i\Lambda)$, which are indecomposable projective and simple modules respectively.
For a $\Lambda$-module $M$, we define the \emph{support} $\operatorname{\mathrm{supp}}M$ of $M$ as $\{i \in Q_{0} \mid Me_{i} \neq 0\}$.

In order to give an interpretation of $s$-torsion pairs by quivers, we need the following notion.

\begin{definition}
Let $I$ be a subset of $Q_0$. We say that $I$ is a \emph{successor-closed subset of $Q_{0}$} if $j \in I$ holds whenever there is an arrow $i \to j$ in $Q$ with $i \in I$.
\end{definition}

Let $\operatorname{\mathsf{succ}}Q$ denote the set of all successor-closed subsets of $Q_{0}$.
It is easily checked that if $I_{1}, I_{2} \in \operatorname{\mathsf{succ}} Q$, then $I_{1} \cap I_{2}, I_{1} \cup I_{2} \in \operatorname{\mathsf{succ}} Q$.
Hence $(\operatorname{\mathsf{succ}}Q, \subseteq)$ forms a lattice.
For $I \in \operatorname{\mathsf{succ}}Q$, let $Q'$ (respectively, $Q''$) be a full subquiver of $Q$ whose vertex set is $I$ (respectively, $Q_{0} \setminus I$).
Then $\operatorname{\mathsf{mod}}kQ'$ and $\operatorname{\mathsf{mod}}kQ''$ are Serre subcategories of $\operatorname{\mathsf{mod}}kQ$.

Now we describe a relationship between $s$-torsion pairs and successor-closed subsets.
For a pair $t:=(\mathcal{T}, \mathcal{F})$ of subcategories of $\mathcal{C}$, we define $I_{t}:=\{i \in Q_{0} \mid S(i) \in \mathcal{T}\}$.
For a subset $I$ of $Q_{0}$, we define a pair $t_{I}:=(\mathcal{T}_{I}, \mathcal{F}_{I})$ of subcategories of $\mathcal{C}$ as
 $\mathcal{T}_{I}:=\{M \in \operatorname{\mathsf{mod}}\Lambda \mid \operatorname{\mathrm{supp}}M \subseteq I\}$ and $\mathcal{F}_{I}:=\{M \in \operatorname{\mathsf{mod}}\Lambda \mid \operatorname{\mathrm{supp}}M \subseteq Q_{0} \setminus I\}$.
Note that $\mathcal{T}_{I}$ and $\mathcal{F}_{I}$ are Serre subcategories by construction.
We can easily check that $t_{I_{t}}=t$ for each $t \in\stors \mathcal{C}$ and $I_{t_{I}}=I$ for each $I \in \operatorname{\mathsf{succ}}Q$ hold.

\begin{proposition}\label{prop_succl}
Let $\Lambda:=kQ$ be the path algebra of a finite acyclic quiver $Q$.
We regard $\mathcal{C}=\operatorname{\mathsf{mod}}\Lambda$ as the extriangulated category with the negative first extension by Example \ref{ex_fingldim}.
Then there exist mutually inverse isomorphisms of posets
\[
\begin{tikzcd}
\stors \mathcal{C} \rar[shift left, "I_{(-)}"] & \operatorname{\mathsf{succ}}Q. \lar[shift left, "t_{(-)}"]
\end{tikzcd}
\]
\end{proposition}

\begin{proof}
Since $t_{(-)}$ and $I_{(-)}$ are mutually inverse to each other and order-preserving, it is enough to show that these maps are well-defined.

Let $t = (\mathcal{T},\mathcal{F})$ be an $s$-torsion pair in $\mathcal{C}$.
We show that $I_{t} \in \operatorname{\mathsf{succ}}Q$.
Assume that there exists an arrow $i \to j$ in $Q$ such that $i \in I_{t}$.
Then $S(i) \in \mathcal{T}$.
By the proof of Proposition \ref{prop_quiverserre}(2)$\Rightarrow$(1), $\mathcal{T}$ is closed under projective covers.
Thus $P(i) \in \mathcal{T}$.
On the other hand, since there exists an arrow $i \to j$, we obtain that $S(j)$ is a composition factor of $P(i)$.
By Corollary \ref{cor_stros_serre}, $\mathcal{T}$ is a Serre subcategory, and hence $S(j) \in \mathcal{T}$.

Let $I \in \operatorname{\mathsf{succ}}Q$.
We show that $t_{I}=(\mathcal{T}_{I}, \mathcal{F}_{I}) \in \stors \mathcal{C}$.
Since Serre subcategories $\mathcal{T}_{I}$ and $\mathcal{F}_{I}$ do not share any simple modules, we have $\operatorname{Hom}_{\Lambda}(\mathcal{T}_{I}, \mathcal{F}_{I})=0$.
By Corollary \ref{cor_stros_serre}, it is enough to show that $\mathcal{C} \subseteq \mathcal{T}_{I} \ast \mathcal{F}_{I}$.
Let $M\in \mathcal{C}$ be an arbitrary module and $tM$ the trace of $\mathcal{T}_{I}$ in $M$, that is, it is a submodule of $M$ generated by all homomorphic images of any $T\in \mathcal{T}_{I}$ in $M$.
Then $tM\in \mathcal{T}_{I}$ holds.
By $I \in \operatorname{\mathsf{succ}}Q$, we have $P_{I}:=\oplus_{i\in I}P(i)\in \mathcal{T}_{I}$.
Since the inclusion $tM\to M$ is a right $\mathcal{T}_{I}$-approximation, we obtain that $\operatorname{Hom}_{\Lambda}(P_{I},M/tM)=0$, and hence $M/tM\in \mathcal{F}_{I}$.
This implies that $M\in \mathcal{T}_{I}\ast \mathcal{F}_{I}$.
\end{proof}

Now, we give an interpretation of Theorem \ref{mainthm} in terms of successor-closed subsets.
Let $I_{1}, I_{2} \in \operatorname{\mathsf{succ}} Q$ satisfying $I_1 \subseteq I_2$.
We put
\begin{align}
\operatorname{\mathsf{succ}}[I_{1}, I_{2}]:=\{I \in \operatorname{\mathsf{succ}}Q \mid I_{1} \subseteq I \subseteq I_{2}\}, \notag
\end{align}
and let $Q_{[I_1,I_2]}$ be the full subquiver of $Q$ whose vertex set is $I_2 \setminus I_1$.

\begin{proposition}\label{prop_succl2}
Let $Q$ be a finite acyclic quiver and $I_{1}, I_{2} \in \operatorname{\mathsf{succ}} Q$ with $I_1 \subseteq I_2$.
Then there exist mutually inverse isomorphisms of posets
\[
\begin{tikzcd}
{\operatorname{\mathsf{succ}}[I_{1}, I_{2}]} \rar[shift left, "\varphi"] & \operatorname{\mathsf{succ}}Q_{[I_{1}, I_{2}]}, \lar[shift left, "\psi"]
\end{tikzcd}
\]
where $\varphi(I):=I \setminus I_{1}$ and $\psi(J):=I_{1}\cup J$.
\end{proposition}

\begin{proof}
Since $I_{1} \subseteq I_{2} \in \operatorname{\mathsf{succ}}Q$, we obtain that
\begin{align}
\mathcal{H}:=\mathcal{H}_{[t_{I_1},t_{I_2}]} = \mathcal{T}_{I_2} \cap \mathcal{F}_{I_1}=\{M \in \operatorname{\mathsf{mod}} \Lambda \mid \operatorname{\mathsf{supp}}M \subseteq I_{2} \setminus I_{1}\} \notag
\end{align}
is a Serre subcategory.
Thus $\mathcal{H}$ is filtered by simple modules whose supports are contained in $I_{2} \setminus I_{1}$.
Hence this is equivalent to the module category of $k Q_{[I_{1}, I_{2}]}$.
By Theorem \ref{mainthm} and Proposition \ref{prop_succl}, we have the following mutually inverse isomorphisms of posets.
\[
\begin{tikzcd}
\stors [t_{I_{1}}, t_{I_{2}}] \rar[shift left, "\Phi"] \dar[shift left, "I_{(-)}"]& \stors (\operatorname{\mathsf{mod}}kQ_{[I_{1},I_{2}]})\lar[shift left, "\Psi"]\dar[shift left, "I_{(-)}"]\\
\operatorname{\mathsf{succ}}[I_{1}, I_{2}] \uar[shift left, "t_{(-)}"] & \operatorname{\mathsf{succ}}Q_{[I_{1}, I_{2}]}\uar[shift left, "t_{(-)}"]
\end{tikzcd}
\]
Let $\varphi:=I_{(-)} \Phi t_{(-)}$ and $\psi:=I_{(-)} \Psi t_{(-)}$.
Then $\varphi$ and $\psi$ are mutually inverse isomorphisms of posets.
Moreover, we obtain that $\varphi(I)=I_{\Phi(t_{I})}=I \setminus I_{1}$ for each $I \in \operatorname{\mathsf{succ}}[I_{1}, I_{2}]$ and $\psi(J)=I_{\Psi(t_{J})}=I_{1} \cup J$ for each $J \in \operatorname{\mathsf{succ}}Q_{[I_{1}, I_{2}]}$.
This finishes the proof.
\end{proof}

Note that since the above proposition is purely combinatorial, one can directly prove it.
Thus Theorem \ref{mainthm} gives a categorical interpretation of this combinatorial statement.

We finish this subsection with giving a concrete example.

\begin{example}
Let $Q$ be a quiver $1 \rightarrow 2 \leftarrow 3 \leftarrow 4$.
Then
\begin{align}
\operatorname{\mathsf{succ}}Q=\{\emptyset, \{2\}, \{1,2\}, \{2,3\}, \{1,2,3\}, \{2,3,4\}, Q_{0}\}.\notag
\end{align}
For each $I \in \operatorname{\mathsf{succ}}Q$, we write the corresponding $s$-torsion pair $(\mathcal{T}_{I}, \mathcal{F}_{I})$ in the Auslander--Reiten quiver of $\operatorname{\mathsf{mod}} kQ$, where the black vertices are $\mathcal{T}_{I}$ and the white vertices are $\mathcal{F}_{I}$.
The following table gives the isomorphism in Proposition \ref{prop_succl}.

\begin{picture}(420,300)(0,0)

\put(30,175){\begin{tikzpicture}[scale = 1]
\node (3) at (3,2) {$3$};
\node (32) at (1,0) {$\substack{3 \\ 2}$};
\node (123) at (2,1) {$\substack{1\,3 \\ 2}$};
\node (2) at (0,1) {$2$};
\node (12) at (1,2) {$\substack{1 \\ 2}$};
\node (1432) at (3,0) {$\substack{\mathrm{~} \, 4\\ 1\, 3 \\ 2}$};
\node (4) at (5,2) {$4$};
\node (43) at (4,1) {$\substack{4 \\ 3}$};
\node (432) at (2,-1) {$\substack{4 \\ 3 \\ 2}$};
\node (1) at (4,-1) {$1$};
\draw[->] (2) -- (12);
\draw[->] (2) -- (32);
\draw[->] (32) -- (123);
\draw[->] (12) -- (123);
\draw[->] (123) -- (3);
\draw[->] (123) -- (1432);
\draw[->] (32) -- (432);
\draw[->] (432) -- (1432);
\draw[->] (1432) -- (1);
\draw[->] (3) -- (43);
\draw[->] (1432) -- (43);
\draw[->] (43) -- (4);
\end{tikzpicture}}
\put(10,165){The Auslander--Reiten quiver of $\operatorname{\mathsf{mod}} kQ$}

\put(70,25){\begin{tikzpicture}[scale = 1]
\node (A) at (2,2) {A};
\node (B) at (1,1) {B};
\node (C) at (3,1) {C};
\node (D) at (1,0) {D};
\node (E) at (3,0) {E};
\node (F) at (2,-1) {F};
\node (G) at (2,-2) {G};
\draw[->] (A) -- (B);
\draw[->] (A) -- (C);
\draw[->] (B) -- (D);
\draw[->] (C) -- (E);
\draw[->] (B) -- (E);
\draw[->] (D) -- (F);
\draw[->] (E) -- (F);
\draw[->] (F) -- (G);
\end{tikzpicture}}
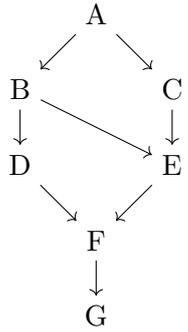
\put(65,10){The Hasse quiver}

\put(250,150){\begin{tabular}{c|c|c}
&$\operatorname{\mathsf{succ}}Q$&$\stors \mathcal{C}$ \\ \hline \hline
A&$Q_{0}$&\begin{tikzpicture}[baseline={([yshift=-.5ex]current bounding box.center)}, scale=0.35,  every node/.style={scale=0.5}]
\node (2) at (0,2) [black] {};
\node (12) at (1,3) [black] {};
\node (23) at (1,1) [black] {};
\node (123) at (2,2) [black] {};
\node (234) at (2,0) [black] {};
\node (3) at (3,3) [black] {};
\node (1234) at (3,1) [black] {};
\node (34) at (4,2) [black] {};
\node (1) at (4,0) [black] {};
\node (4) at (5,3) [black] {};
\draw[->] (2) -- (12);
\draw[->] (2) -- (23);
\draw[->] (12) -- (123);
\draw[->] (23) -- (123);
\draw[->] (23) -- (234);
\draw[->] (123) -- (3);
\draw[->] (123) -- (1234);
\draw[->] (234) -- (1234);
\draw[->] (3) -- (34);
\draw[->] (1234) -- (34);
\draw[->] (1234) -- (1);
\draw[->] (34) -- (4);
\node at (0,-.2) {};
\node at (0,3.2) {};
\end{tikzpicture} \\ \hline
B&$\{ 1,2,3 \}$&\begin{tikzpicture}[baseline={([yshift=-.5ex]current bounding box.center)}, scale=0.35,  every node/.style={scale=0.5}]
\node (2) at (0,2) [black] {};
\node (12) at (1,3) [black] {};
\node (23) at (1,1) [black] {};
\node (123) at (2,2) [black] {};
\node (234) at (2,0) [] {};
\node (3) at (3,3) [black] {};
\node (1234) at (3,1) [] {};
\node (34) at (4,2) [] {};
\node (1) at (4,0) [black] {};
\node (4) at (5,3) [white] {};
\draw[->] (2) -- (12);
\draw[->] (2) -- (23);
\draw[->] (12) -- (123);
\draw[->] (23) -- (123);
\draw[->] (23) -- (234);
\draw[->] (123) -- (3);
\draw[->] (123) -- (1234);
\draw[->] (234) -- (1234);
\draw[->] (3) -- (34);
\draw[->] (1234) -- (34);
\draw[->] (1234) -- (1);
\draw[->] (34) -- (4);
\node at (0,-.2) {};
\node at (0,3.2) {};
\end{tikzpicture} \\ \hline
C&$\{ 2,3,4 \}$&\begin{tikzpicture}[baseline={([yshift=-.5ex]current bounding box.center)}, scale=0.35,  every node/.style={scale=0.5}]
\node (2) at (0,2) [black] {};
\node (12) at (1,3) [] {};
\node (23) at (1,1) [black] {};
\node (123) at (2,2) [] {};
\node (234) at (2,0) [black] {};
\node (3) at (3,3) [black] {};
\node (1234) at (3,1) [] {};
\node (34) at (4,2) [black] {};
\node (1) at (4,0) [white] {};
\node (4) at (5,3) [black] {};
\draw[->] (2) -- (12);
\draw[->] (2) -- (23);
\draw[->] (12) -- (123);
\draw[->] (23) -- (123);
\draw[->] (23) -- (234);
\draw[->] (123) -- (3);
\draw[->] (123) -- (1234);
\draw[->] (234) -- (1234);
\draw[->] (3) -- (34);
\draw[->] (1234) -- (34);
\draw[->] (1234) -- (1);
\draw[->] (34) -- (4);
\node at (0,-.2) {};
\node at (0,3.2) {};
\end{tikzpicture} \\ \hline
D&$\{ 1,2 \}$&\begin{tikzpicture}[baseline={([yshift=-.5ex]current bounding box.center)}, scale=0.35,  every node/.style={scale=0.5}]
\node (2) at (0,2) [black] {};
\node (12) at (1,3) [black] {};
\node (23) at (1,1) [] {};
\node (123) at (2,2) [] {};
\node (234) at (2,0) [] {};
\node (3) at (3,3) [white] {};
\node (1234) at (3,1) [] {};
\node (34) at (4,2) [white] {};
\node (1) at (4,0) [black] {};
\node (4) at (5,3) [white] {};
\draw[->] (2) -- (12);
\draw[->] (2) -- (23);
\draw[->] (12) -- (123);
\draw[->] (23) -- (123);
\draw[->] (23) -- (234);
\draw[->] (123) -- (3);
\draw[->] (123) -- (1234);
\draw[->] (234) -- (1234);
\draw[->] (3) -- (34);
\draw[->] (1234) -- (34);
\draw[->] (1234) -- (1);
\draw[->] (34) -- (4);
\node at (0,-.2) {};
\node at (0,3.2) {};
\end{tikzpicture} \\ \hline
E&$\{ 2,3 \}$&\begin{tikzpicture}[baseline={([yshift=-.5ex]current bounding box.center)}, scale=0.35,  every node/.style={scale=0.5}]
\node (2) at (0,2) [black] {};
\node (12) at (1,3) [] {};
\node (23) at (1,1) [black] {};
\node (123) at (2,2) [] {};
\node (234) at (2,0) [] {};
\node (3) at (3,3) [black] {};
\node (1234) at (3,1) [] {};
\node (34) at (4,2) [] {};
\node (1) at (4,0) [white] {};
\node (4) at (5,3) [white] {};
\draw[->] (2) -- (12);
\draw[->] (2) -- (23);
\draw[->] (12) -- (123);
\draw[->] (23) -- (123);
\draw[->] (23) -- (234);
\draw[->] (123) -- (3);
\draw[->] (123) -- (1234);
\draw[->] (234) -- (1234);
\draw[->] (3) -- (34);
\draw[->] (1234) -- (34);
\draw[->] (1234) -- (1);
\draw[->] (34) -- (4);
\node at (0,-.2) {};
\node at (0,3.2) {};
\end{tikzpicture} \\ \hline
F&$\{ 2 \}$&\begin{tikzpicture}[baseline={([yshift=-.5ex]current bounding box.center)}, scale=0.35,  every node/.style={scale=0.5}]
\node (2) at (0,2) [black] {};
\node (12) at (1,3) [] {};
\node (23) at (1,1) [] {};
\node (123) at (2,2) [] {};
\node (234) at (2,0) [] {};
\node (3) at (3,3) [white] {};
\node (1234) at (3,1) [] {};
\node (34) at (4,2) [white] {};
\node (1) at (4,0) [white] {};
\node (4) at (5,3) [white] {};
\draw[->] (2) -- (12);
\draw[->] (2) -- (23);
\draw[->] (12) -- (123);
\draw[->] (23) -- (123);
\draw[->] (23) -- (234);
\draw[->] (123) -- (3);
\draw[->] (123) -- (1234);
\draw[->] (234) -- (1234);
\draw[->] (3) -- (34);
\draw[->] (1234) -- (34);
\draw[->] (1234) -- (1);
\draw[->] (34) -- (4);
\node at (0,-.2) {};
\node at (0,3.2) {};
\end{tikzpicture} \\ \hline
G&$\emptyset$&\begin{tikzpicture}[baseline={([yshift=-.5ex]current bounding box.center)}, scale=0.35,  every node/.style={scale=0.5}]
\node (2) at (0,2) [white] {};
\node (12) at (1,3) [white] {};
\node (23) at (1,1) [white] {};
\node (123) at (2,2) [white] {};
\node (234) at (2,0) [white] {};
\node (3) at (3,3) [white] {};
\node (1234) at (3,1) [white] {};
\node (34) at (4,2) [white] {};
\node (1) at (4,0) [white] {};
\node (4) at (5,3) [white] {};
\draw[->] (2) -- (12);
\draw[->] (2) -- (23);
\draw[->] (12) -- (123);
\draw[->] (23) -- (123);
\draw[->] (23) -- (234);
\draw[->] (123) -- (3);
\draw[->] (123) -- (1234);
\draw[->] (234) -- (1234);
\draw[->] (3) -- (34);
\draw[->] (1234) -- (34);
\draw[->] (1234) -- (1);
\draw[->] (34) -- (4);
\node at (0,-.2) {};
\node at (0,3.2) {};
\end{tikzpicture} \\ \hline
\end{tabular}}
\end{picture}

Let $I_{1}:=\{2\}, I_{2}:=\{2,3,4\} \in \operatorname{\mathsf{succ}}Q$.
Then $Q_{[I_{1},I_{2}]}=(3 \leftarrow 4)$.
Thus we can easily check that $\operatorname{\mathsf{succ}}[I_{1}, I_{2}]=\{I_{1}, \{2,3\}, I_{2}\}$ and $\operatorname{\mathsf{succ}}Q_{[I_{1}, I_{2}]}=\{\emptyset, \{3\}, \{3, 4\} \}$.
Hence we obtain the isomorphisms in Proposition \ref{prop_succl2}.
\end{example}

\subsection*{Acknowledgements}
The second author would like to thank Shunya Saito for helpful discussions on periodic derived categories in Remark \ref{rem_periodic}.

\end{document}